\documentclass[final,1p,times]{elsarticle}
\usepackage{url}
\usepackage{amssymb}
\usepackage{amsmath}
\usepackage{stmaryrd}
\usepackage{siunitx}
\usepackage{commath}
\usepackage{enumitem}
\usepackage[caption=false]{subfig}
\usepackage{cleveref}
\journal{Journal of Computational and Applied Mathematics}
\makeatletter
\newcommand{\tnorm}{\@ifstar\@tnorms\@tnorm}
\newcommand{\@tnorms}[1]{%
  \left|\mkern-1.5mu\left|\mkern-1.5mu\left|
   #1
  \right|\mkern-1.5mu\right|\mkern-1.5mu\right|
}
\newcommand{\@tnorm}[2][]{%
  \mathopen{#1|\mkern-1.5mu#1|\mkern-1.5mu#1|}
  #2
  \mathclose{#1|\mkern-1.5mu#1|\mkern-1.5mu#1|}
}
\makeatother

\newcommand{\jump}[1]{\llbracket #1 \rrbracket}
\newtheorem{theorem}{Theorem}
\newtheorem{lemma}{Lemma}

\newtheorem{proposition}{Proposition}
\newdefinition{remark}{Remark}
\newproof{proof}{Proof}
\begin{document}
\begin{frontmatter}
\title{An embedded--hybridized discontinuous Galerkin method for
  the coupled Stokes--Darcy system}
\author[AC]{Aycil Cesmelioglu\fnref{label1}}
\ead{cesmelio@oakland.edu}
\fntext[label1]{\url{https://orcid.org/0000-0001-8057-6349}}
\address[AC]{Department of Mathematics and Statistics, Oakland
  University, Michigan, USA}

\author[SR]{Sander Rhebergen\corref{cor1}\fnref{label2}}
\ead{srheberg@uwaterloo.ca}
\fntext[label2]{\url{https://orcid.org/0000-0001-6036-0356}}
\address[SR]{Department of Applied Mathematics, University of
  Waterloo, Canada}

\author[GNW]{Garth N. Wells\fnref{label3}}
\ead{gnw20@cam.ac.uk}
\fntext[label3]{\url{https://orcid.org/0000-0001-5291-7951}}
\address[GNW]{Department of Engineering, University of Cambridge,
  United Kingdom}

\begin{abstract}
  We introduce an embedded--hybridized discontinuous Galerkin (EDG--HDG)
  method for the coupled Stokes--Darcy system. This EDG--HDG method is a
  pointwise mass-conserving discretization resulting in a
  divergence-conforming velocity field on the whole domain. In the
  proposed scheme, coupling between the Stokes and Darcy domains is
  achieved naturally through the EDG--HDG facet variables. \emph{A
  priori} error analysis shows optimal convergence rates, and that the
  velocity error does not depend on the pressure. The error analysis is
  verified through numerical examples on unstructured grids for
  different orders of polynomial approximation.
\end{abstract}
\end{frontmatter}
\section{Introduction}
\label{sec:introduction}

Modelling adjacent free flow and porous media flow is important for a
range of applications, e.g., transport of drugs via blood flow in
vessels in biomedical engineering, and transport of pollutants via
surface/groundwater flow in environmental engineering. The problem can
be stated as a system of partial differential equations, with free flow
governed by the Stokes equations and porous media flow governed by
Darcy's equations. The interactions at the boundary between the free
flow and porous media flow regions were specified by \cite{Beavers:1967,
Saffman:1971}, and were mathematically justified in~\cite{Mikelic:2000}.
We refer to \cite{Discacciati:2009} for an overview of the model.

Well-posedness of the weak formulation of the Stokes--Darcy problem can
be found in \cite{Discacciati:2002} for the primal form, and in
\cite{Layton:2002} for the primal--mixed form. Many different finite
element and mixed finite element methods have been proposed to
discretize the Stokes--Darcy problem for both formulations,
e.g.~\cite{Discacciati:2002, Layton:2002, Burman:2005, Cao:2010,
Gatica:2009, Camano:2015, Marquez:2015}. Other devised finite element
methods include discontinuous Galerkin (DG) methods
\cite{Cesmelioglu:2009, Girault:2009, Girault:2014, Lipnikov:2014,
Riviere:2005, RiviereYotov:2005}, hybridizable discontinuous Galerkin
(HDG) methods \cite{Egger:2013, Fu:2018, Gatica:2017, Igreja:2018}, weak
Galerkin methods (WG) \cite{Chen:2016}, and weak virtual element methods
(WVEM) \cite{Wang:2019}.

We develop a numerical scheme for which the velocity field is
divergence-conforming on the whole domain and for which mass is
conserved pointwise. Finite element methods that satisfy these
properties were proposed in \cite{Girault:2014, Kanschat:2010} where
they are referred to as `strongly conservative'. For the Stokes region
\cite{Girault:2014, Kanschat:2010} used a divergence-conforming DG
space for the velocity and a standard DG space for the pressure. In
the Darcy region they used a mixed finite element method. It is
well-known, however, that DG methods can be expensive due to the large
number of degrees-of-freedom on a given mesh compared to other
methods. To reduce the number of globally coupled degrees-of-freedom,
\cite{Fu:2018} proposed an HDG method for the Stokes region using a
divergence-conforming finite element space for the velocity. Their
method results in less globally coupled degrees-of-freedom compared to
standard HDG methods as they only enforce continuity of the tangential
direction of the facet velocity. Additionally, to reduce the problem
size even further, they applied the `projected jumps' method in which
the polynomial degree of the tangential facet velocity is reduced by
one compared to the cell velocity approximation (see also
\cite{Lehrenfeld:2016}).

In this paper we propose an embedded--hybridized discontinuous Galerkin
(EDG--HDG) finite element method of the primal--mixed formulation of the
Stokes--Darcy problem on the whole domain. The EDG--HDG method uses a
continuous trace velocity approximation and a discontinuous trace
pressure approximation.  Due to the continuous trace velocity
approximation, the number of globally coupled degrees-of-freedom of the
EDG--HDG method is fewer than for a traditional HDG method. However, the
main motivation for an EDG--HDG discretization is not that the problem
size is smaller, but that `continuous' discretizations are generally
better suited to fast iterative solvers. This was demonstrated for the
Stokes problem in \cite{Rhebergen:2019}, where CPU time and iteration
count to convergence was reduced significantly compared to a hybridized
method using only discontinuous facet approximations. We will show that
the EDG--HDG method proposed in this paper is pointwise mass-conserving
and that the resulting velocity field is divergence-conforming. We
present furthermore an analysis of the proposed EDG--HDG method for the
Stokes--Darcy problem, proving well-posedness, and optimal \emph{a
priori} error estimates.

The remainder of this paper is organized as follows. In
\cref{sec:stokesdarcy} we briefly introduce the coupled Stokes--Darcy
problem. The EDG--HDG method to this problem is presented in
\cref{sec:hdgstokesdarcy}. Consistency, continuity and well-posedness
are shown in \cref{sec:consiscontwell} while the main results of this
paper, an \emph{a priori} error analysis, is presented in
\cref{sec:erroranalysisstokesdarcy}. Numerical simulations support our
theoretical results in \cref{sec:numerical_examples}, and conclusions
are drawn in \cref{sec:conclusions}.

\section{The Stokes--Darcy system}
\label{sec:stokesdarcy}

Let $\Omega \subset \mathbb{R}^{\rm dim}$ be a bounded polygonal domain
with ${\rm dim}=2,3$, boundary $\partial\Omega$ and boundary outward
unit normal~$n$. We assume that $\Omega$ is divided into two
non-overlapping regions, $\Omega^s$ and $\Omega^d$, such that $\Omega =
\Omega^s \cup \Omega^d$ and $\Omega^s$ and $\Omega^d$ are a union of
polygonal subdomains. We denote the polygonal interface between
$\Omega^s$ and $\Omega^d$ by $\Gamma^I$. Furthermore, we define the
external boundary of $\Omega^s$ by $\Gamma^s := \partial \Omega \cap
\partial \Omega^s$, and the external boundary of $\Omega^d$ by $\Gamma^d
:= \partial \Omega \cap \partial \Omega^d$. See
\cref{fig:domaindescription}.

\begin{figure}
  \centering
  \includegraphics[width=0.4\textwidth]{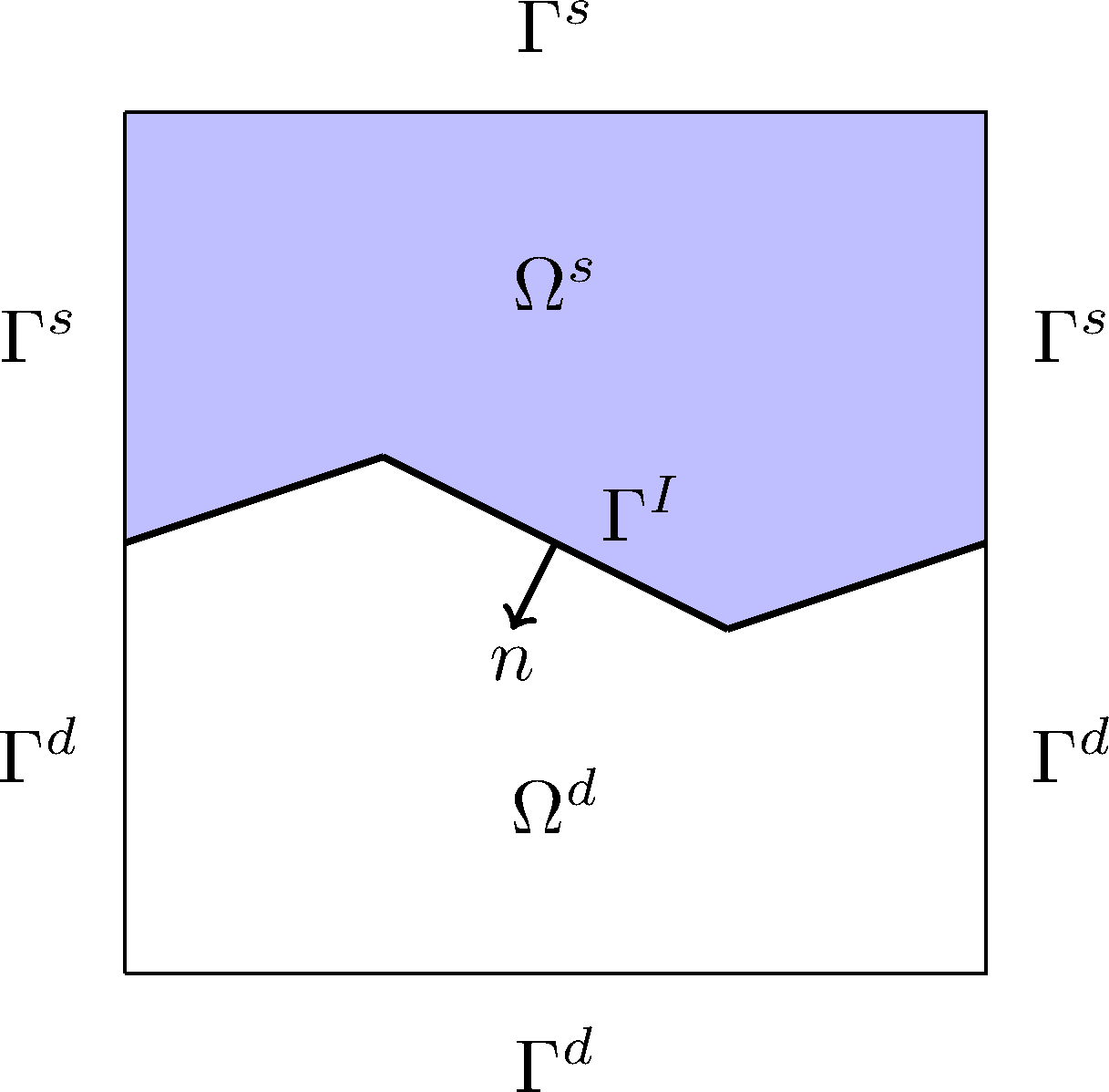}
  \caption{A depiction of the Stokes--Darcy region.}
  \label{fig:domaindescription}
\end{figure}

Given the kinematic viscosity $\mu\in\mathbb{R}^+$, forcing term $f^s
:\Omega^s \to \mathbb{R}^{\rm dim}$, permeability constant $\kappa \in
\mathbb{R}^+$ and the source/sink term $f^d :\Omega^d \to \mathbb{R}$,
the Stokes--Darcy system for the velocity field $u :\Omega \to
\mathbb{R}^{\rm dim}$ and pressure $p : \Omega \to \mathbb{R}$ is given
by
\begin{subequations}
  \label{eq:system}
  \begin{align}
    \label{eq:momentum}
    -\nabla\cdot 2\mu\varepsilon(u) + \nabla p &= f^s & & \text{in}\ \Omega^s,
    \\
    \label{eq:d_velocity}
    \kappa^{-1}u + \nabla p &= 0 & & \text{in}\ \Omega^d,
    \\
    \label{eq:mass}
    - \nabla\cdot u &= \chi^d f^d & & \text{in}\ \Omega,
    \\
    \label{eq:bc_s}
    u &= 0 & & \text{on}\ \Gamma^s,\\
    \label{eq:bc_d}
    u\cdot n &= 0 & & \text{on}\ \Gamma^d,
  \end{align}
\end{subequations}
where $\varepsilon(u) := \del{\nabla u + (\nabla u)^T}/2$ is the strain
rate tensor and $\chi^d$ is the characteristic function that has the
value 1 in $\Omega^d$ and 0 in $\Omega^s$. We will also frequently
denote the velocity and pressure in $\Omega^j$ by $u^j$ and $p^j$,
respectively, for $j = s,d$.

Let $n$ denote the unit normal vector on the interface between the two
domains, $\Gamma^I$, pointing outwards from~$\Omega^s$. On the interface
$\Gamma^I$ we prescribe
\begin{subequations}
  \label{eq:interface}
  \begin{align}
    \label{eq:bc_I_u}
    u^s\cdot n &= u^d\cdot n & & \text{on}\ \Gamma^I,
    \\
    \label{eq:bc_I_p}
    p^s - 2\mu\varepsilon(u^s)n\cdot n &= p^d & & \text{on}\ \Gamma^I,
    \\
    \label{eq:bc_I_slip}
    -2\mu\del{\varepsilon(u^s)n}^t &= \alpha \kappa^{-1/2}(u^s)^t & & \text{on}\ \Gamma^I,
  \end{align}
\end{subequations}
where the tangential component of a vector $w$ is denoted by $(w)^t := w
- (w\cdot n)n$. \Cref{eq:bc_I_slip} is the Beavers--Joseph--Saffman law
\cite{Beavers:1967, Saffman:1971}, where $\alpha > 0$ is an
experimentally determined parameter.

\section{The embedded--hybridized discontinuous Galerkin method}
\label{sec:hdgstokesdarcy}

We present now an embedded--hybridized discontinuous Galerkin (EDG--HDG)
method for the Stokes--Darcy system \cref{eq:system,eq:interface}, and
establish some of its key properties.

\subsection{Preliminaries}

For $j = s, d$, let $\mathcal{T}^j := \cbr{K}$ be a triangulation of
$\Omega^j$ into non-overlapping cells $K$. For brevity, we consider the
case of matching meshes at the interface~$\Gamma^I$. Furthermore, let
$\mathcal{T} := \mathcal{T}^s \cup \mathcal{T}^d$. The diameter of a
cell $K$ is denoted by $h_K$ and $h$ denotes the maximum diameter over
all cells. The outward unit normal vector on the boundary of a cell,
$\partial K$, is denoted by~$n$. An interior facet $F$ is shared by
adjacent cells, $K^+$ and $K^-$, while a boundary facet is part of
$\partial K$ that lies on~$\partial\Omega$. The set and union of all
facets are denoted by $\mathcal{F} := \cbr{F}$ and $\Gamma_0$,
respectively. By $\mathcal{F}^I$ we denote the set of all facets that
lie on $\Gamma^I$. For $j=s,d$, we denote by $\mathcal{F}^j$ the set of
all facets that lie in $\overline{\Omega}^j$. Finally, for $j = s, d$,
we denote the union of all facets in $\overline{\Omega}^j$
by~$\Gamma_0^j$.

We consider the following discontinuous Galerkin finite element function
spaces on~$\Omega$,
\begin{equation}
  \label{eq:DGcellspaces}
  \begin{split}
    V_h &:= \cbr[1]{v_h \in \sbr[0]{L^2(\Omega)}^{\rm dim} : \ v_h \in
      \sbr[0]{P_k(K)}^{\rm dim} \ \forall\ K \in \mathcal{T}},
    \\
    Q_h &:= \cbr[1]{q_h \in L^2(\Omega) : \ q_h \in P_{k-1}(K) \
      \forall \ K \in \mathcal{T}}\cap L^2_0(\Omega),
    \\
    Q_h^j &:= \cbr[1]{q_h \in L^2(\Omega^j) : \ q_h \in P_{k-1}(K) \
      \forall \ K \in \mathcal{T}^j},\quad j=s,d,
  \end{split}
\end{equation}
where $P_k(D)$ denotes the space of polynomials of degree $k$ on
domain~$D$ and $L^2_0(\Omega) := \{q \in L^2(\Omega) : \int_{\Omega} q
\dif x = 0\}$. On $\Gamma_0^s$ and $\Gamma_0^d$, we consider the finite
element spaces:
\begin{equation}
  \label{eq:DGfacetspaces}
  \begin{split}
    \bar{V}_h &:= \cbr[1]{\bar{v}_h \in \sbr[0]{L^2(\Gamma_0^s)}^{\rm dim}:\
      \bar{v}_h \in \sbr[0]{P_{k}(F)}^{\rm dim}\ \forall\ F \in \mathcal{F}^s,\
      \bar{v}_h = 0 \ \mbox{on}\ \Gamma^s} \cap \sbr[1]{C^0(\Gamma_0^s)}^{\rm dim},
    \\
    \bar{Q}_h^j &:= \cbr[1]{\bar{q}_h^j \in L^2(\Gamma_0^j) : \ \bar{q}_h^j
      \in P_{k}(F) \ \forall\ F \in \mathcal{F}^j},\quad j=s,d.
  \end{split}
\end{equation}
Note that $\bar{q}_h^s \in \bar{Q}_h^s$ and $\bar{q}_h^d \in
\bar{Q}_h^d$ do not necessarily coincide on the interface~$\Gamma^I$.
Note also that functions in $\bar{V}_h$ are continuous on $\Gamma_0^s$,
while functions in $\bar{Q}_h^j$ are discontinuous on $\Gamma_0^j$, for
$j=s,d$.

For notational purposes, we introduce the spaces $\boldsymbol{V}_h :=
V_h \times \bar{V}_h$, $\boldsymbol{Q}_h := Q_h \times \bar{Q}_h^s
\times \bar{Q}_h^d$ and $\boldsymbol{Q}_h^{j} := Q_h^j \times
\bar{Q}_h^j$ for $j=s,d$. Function pairs in $\boldsymbol{V}_h$,
$\boldsymbol{Q}_h$ and $\boldsymbol{Q}_h^{j}$, for $j=s,d$, will be
denoted by $\boldsymbol{v}_h := (v_h, \bar{v}_h) \in \boldsymbol{V}_h$,
$\boldsymbol{q}_h := (q_h, \bar{q}_h^s, \bar{q}_h^d) \in
\boldsymbol{Q}_h$ and $\boldsymbol{q}_h^{j} := (q_h, \bar{q}_h^j) \in
\boldsymbol{Q}_h^{j}$. Furthermore, we set $\boldsymbol{X}_h :=
\boldsymbol{V}_h \times \boldsymbol{Q}_h$.

\subsection{Method}

The discrete form for the Stokes--Darcy system in
\cref{eq:system,eq:interface} reads: given the forcing term $f^s \in
\sbr[0]{L^2(\Omega^s)}^{\rm dim}$, the source/sink term $f^d \in
L^2(\Omega^d)$, the kinematic viscosity $\mu\in\mathbb{R}^+$ and the
permeability $\kappa\in\mathbb{R}^+$, find $(\boldsymbol{u}_h,
\boldsymbol{p}_h) \in \boldsymbol{X}_h$ such that
\begin{equation}
  \label{eq:hdgwf}
  B_h( (\boldsymbol{u}_h, \boldsymbol{p}_h), (\boldsymbol{v}_h, \boldsymbol{q}_h) )
  = \int_{\Omega^s} f^s\cdot v_h \dif x + \int_{\Omega^d} f^dq_h\dif x
  \quad \forall (\boldsymbol{v}_h, \boldsymbol{q}_h) \in \boldsymbol{X}_h,
\end{equation}
where
\begin{multline}
  B_h( (\boldsymbol{u}_h, \boldsymbol{p}_h), (\boldsymbol{v}_h, \boldsymbol{q}_h) )
  :=
  a_h(\boldsymbol{u}_h, \boldsymbol{v}_h)
  + \sum_{j=s,d} \del{b_h^j(\boldsymbol{p}_h^j, v_h) + b_h^{I,j}(\bar{p}_h^j, \bar{v}_h)}
  \\
  + \sum_{j=s,d} \del{b_h^j(\boldsymbol{q}_h^j, u_h) + b_h^{I,j}(\bar{q}_h^j, \bar{u}_h)}.
\end{multline}
The bilinear form $a_h(\cdot, \cdot)$ is defined as
\begin{equation}
  \label{eq:def_ah}
  a_h(\boldsymbol{u}_h, \boldsymbol{v}_h)
  := a_h^s(\boldsymbol{u}_h, \boldsymbol{v}_h)
  + a_h^d(u_h, v_h) + a_h^I(\bar{u}_h, \bar{v}_h),
\end{equation}
where
\begin{subequations}
  \begin{align}
    \label{eq:ah_s}
    a_h^s(\boldsymbol{u}, \boldsymbol{v})
    &:=
    \sum_{K\in\mathcal{T}^s} \int_K 2\mu \varepsilon(u) : \varepsilon(v) \dif x
    +\sum_{K\in\mathcal{T}^s} \int_{\partial K} \frac{2\beta\mu}{h_K}(u-\bar{u}) \cdot (v-\bar{v}) \dif s
    \\
    \nonumber
    &\ -\sum_{K\in\mathcal{T}^s} \int_{\partial K} 2\mu\varepsilon(u)n^s \cdot (v-\bar{v})\dif s
    -\sum_{K\in\mathcal{T}^s} \int_{\partial K} 2\mu\varepsilon(v)n^s \cdot (u-\bar{u})\dif s,
    \\
    \label{eq:ah_d}
    a_h^d(u, v)
    &:= \int_{\Omega^d} \kappa^{-1} u\cdot v \dif x ,
    \\
    \label{eq:ah_i}
    a_h^I(\bar{u}, \bar{v}) &:= \int_{\Gamma^I}  \alpha\kappa^{-1/2} \bar{u}^t\cdot \bar{v}^t\dif s,
  \end{align}
\end{subequations}
and where $\beta > 0$ is a penalty parameter. The bilinear forms
$b_h^j(\cdot, \cdot)$ and $b_h^{I,j}(\cdot, \cdot)$ are defined as:
\begin{subequations}
  \begin{align}
    \label{eq:bh_j}
    b_h^{j}(\boldsymbol{p}^{j},v )
      &:= -\sum_{K\in\mathcal{T}^j} \int_K p \nabla\cdot v \dif x
              + \sum_{K\in\mathcal{T}^j} \int_{\partial K} \bar{p}^j v\cdot n^j \dif s,
    \\
    \label{eq:bh_Ij}
    b_h^{I,j}(\bar{p}^j, \bar{v} )
      &:= -\int_{\Gamma^I}\bar{p}^j\bar{v}\cdot n^j \dif s.
  \end{align}
\end{subequations}

\subsection{Properties of the numerical scheme}

Setting $\bar{v}_h = 0$ and $\boldsymbol{q}_h = 0$ in \cref{eq:hdgwf}
demonstrates cell-wise momentum balance \cref{eq:momentum} subject to
weak satisfaction of the boundary condition provided by $\bar{u}_h$, and
a cell-wise statement of Darcy's law \cref{eq:d_velocity} subject to
weak satisfaction of the boundary condition provided by $\bar{p}_h^d$
and a Neumann boundary condition on~$\Gamma^I$. Setting $v_h = 0$ and
$\boldsymbol{q}_h = 0$ in \cref{eq:hdgwf} shows that the formulation
imposes normal continuity weakly across facets of the `numerical' Stokes
stress tensor:
\begin{equation}
  \hat{\sigma}_h^s := 2\mu\varepsilon(u_h^s) - \bar{p}_h^s\mathbb{I}
  - \frac{2\beta\mu}{h_K}(u_h^s - \bar{u}_h^s)\otimes n,
\end{equation}
where $\mathbb{I}$ is the identity tensor. Setting $\boldsymbol{v}_h =
0$ and $\bar{q}_h^s = \bar{q}_h^d = 0$ in \cref{eq:hdgwf} and noting
that $\nabla \cdot u_h\in P_{k-1}(K)$, the numerical scheme imposes
pointwise mass conservation, i.e.,
\begin{equation}
  \label{eq:massconservation-1}
  -\nabla \cdot u_h = \chi^d \Pi_Q f^d \quad \forall x\in K, \ \forall K \in \mathcal{T},
\end{equation}
where $\Pi_Q$ is the standard $L^2$-projection operator onto $Q_h$.
Finally, setting $\boldsymbol{v}_h = 0$, $q_h = 0$ and $\bar{q}_h^l=0$
with $l \ne j$ in \cref{eq:hdgwf} and noting that $\jump{u_h\cdot n},
u_h\cdot n, (u_h-\bar{u}_h)\cdot n\in P_k(F)$ on each $F \in
\mathcal{F}$, we find that $u_h$ is $H(\text{div})$-conforming, i.e.,
\begin{subequations}
  \begin{align}
    \label{eq:massconservation-2}
    \jump{u_h\cdot n}&=0 && \forall x\in F,\ \forall F\in \mathcal{F} \backslash \mathcal{F}^I,
    \\
    \label{eq:massconservation-4}
    u_h\cdot n&=\bar{u}_h\cdot n && \forall x\in F,\ \forall F\in \mathcal{F}^I,
  \end{align}
  \label{eq:massconservations}
\end{subequations}
where $\jump{\cdot}$ is the usual jump operator and $n$ the unit normal
vector on~$F$.

\section{Consistency, continuity and well-posedness for Stokes--Darcy}
\label{sec:consiscontwell}

\subsection{Preliminaries}
\label{ss:ccw_prelim}

To prove consistency, continuity and stability we require extended
function spaces and appropriate norms. We introduce
\begin{equation}
  \begin{split}
    V &:= \cbr[1]{v: v^s\in \sbr[0]{H^2(\Omega^s)}^{\rm dim},\
        v^d \in \sbr[0]{H^1(\Omega^d)}^{\rm dim},\
      v=0\ \text{on}\ \Gamma^s,\ v \cdot n=0 \ \text{on}\ \Gamma^d,
      \ v^s\cdot n = v^d\cdot n\ \text{on}\ \Gamma^I},
    \\
    Q &:= \cbr[1]{q \in L_0^2(\Omega)\ :\ q^s \in H^1(\Omega^s),\ q^d \in H^2(\Omega^d)},
  \end{split}
\end{equation}
and $X := V \times Q$. We let $\bar{V}$ be the trace space of $V$
restricted to $\Gamma_0^s$ and $\bar{Q}$ be the trace space of $Q$
restricted to $\Gamma_0$. We introduce the trace operator $\gamma_V : V
\to \bar{V}$ to restrict functions in $V$ to $\Gamma_0^s$, and the trace
operators $\gamma_{Q^j} : Q^j \to \bar{Q}^j$ to restrict functions in
$Q^j$ to $\Gamma_0^j$. We remark that $\gamma_{Q^s}(p^s) \ne
\gamma_{Q^d}(p^d)$ on the interface~$\Gamma^I$. Where no ambiguity
arises we omit the subscript when using the trace operator. For
notational purposes we also introduce $\boldsymbol{V} := V\times
\bar{V}$, $\boldsymbol{Q} := Q\times \bar{Q}$ and
\begin{equation}
  \boldsymbol{V}(h) := \boldsymbol{V}_h + \boldsymbol{V}, \quad
  \boldsymbol{Q}(h) := \boldsymbol{Q}_h + \boldsymbol{Q}, \quad
  \boldsymbol{X}(h) := \boldsymbol{V}(h) \times \boldsymbol{Q}(h).
\end{equation}
For $j = s, d$ we denote by $\boldsymbol{V}^{j}(h)$ and
$\boldsymbol{Q}^{j}(h)$ the restriction of, respectively,
$\boldsymbol{V}(h)$ and $\boldsymbol{Q}(h)$ to~$\Omega^j$.

We use various norms throughout, which are defined now. On
$\boldsymbol{V}^s(h)$ we define
\begin{equation*}
  \tnorm{\boldsymbol{v}}_{v,s}^2 := \sum_{K\in \mathcal{T}^s}\del[1]{\norm[0]{\nabla v}_K^2
  + h_K^{-1}\norm{v-\bar{v}}^2_{\partial K}},
\end{equation*}
where $\norm{\cdot}_D$ denotes the standard $L^2$-norm on domain~$D$,
and
\begin{equation*}
  \tnorm{\boldsymbol{v}}_{v',s}^2 := \tnorm{\boldsymbol{v}}_{v,s}^2
  + \sum_{K\in \mathcal{T}^s} h_K^2 \envert{v}_{H^2(K)}^2.
\end{equation*}
On $\boldsymbol{V}(h)$ we introduce
\begin{equation*}
  \tnorm{\boldsymbol{v}}_{v}^2
  :=
  \tnorm{\boldsymbol{v}}_{v,s}^2
  + \norm{v}_{\Omega^d}^2
  + \norm[0]{ \bar{v}^t }^2_{\Gamma^I},
\end{equation*}
and
\begin{equation*}
  \tnorm{\boldsymbol{v}}_{v'}^2
  := \tnorm{\boldsymbol{v}}_{v}^2
  + \sum_{K\in \mathcal{T}^s} h_K^2 \envert{v}_{H^2(K)}^2
  =\tnorm{\boldsymbol{v}}_{v',s}^2  + \norm{v}_{\Omega^d}^2
  + \norm[0]{ \bar{v}^t }^2_{\Gamma^I}.
\end{equation*}
Note that the norms $\tnorm{\cdot}_v$ and $\tnorm{\cdot}_{v'}$ are
equivalent on $\boldsymbol{V}_h$, see \cite[eq.~(5.5)]{Wells:2011}.

Finally, for $\boldsymbol{q}^j \in \boldsymbol{Q}^{j}(h)$ with $j = s,
d$ and $\boldsymbol{q} \in \boldsymbol{Q}(h)$, we define
\begin{equation*}
  \tnorm{\boldsymbol{q}^j}_{p,j}^2 := \norm{q}_{\Omega^j}^2
  + \sum_{K\in \mathcal{T}^j} h_K \norm[0]{\bar{q}^j}_{\partial K}^2,
  \qquad   \tnorm{\boldsymbol{q}}_{p}^2
  := \sum_{j=s,d}\tnorm{\boldsymbol{q}^j}^2_{p,j}.
\end{equation*}

We will make use of various standard estimates. In particular, use will
be made of the trace inequalities for $K\in\mathcal{T}$,
\cite[Lemma~1.46, Remark~1.47]{Pietro:book}
\begin{equation}
  \label{eq:trace-1}
  \norm{v}_{\partial K} \le C_{T,1} h_K^{-1/2} \norm{v}_K \quad \forall v\in P_k(K),
\end{equation}
and the following straightforward extensions of the continuous trace
inequality \cite[Theorem~1.6.6]{brenner:book}
\begin{equation}
  \label{eq:trace-2}
  \norm{v}_{\partial K}^2 \le C_{T,2} \del[1]{h_K^{-1} \norm{v}_K^2
  + h_K \norm{v}_{1,K}^2} \quad \forall v\in H^1(K),
\end{equation}
and
\begin{equation}
  \label{eq:continuoustrace}
  \norm{v}_{\Gamma^I} \le C_{c,T} \norm[0]{\nabla v}_{\Omega^s} \quad \forall
  v \in \cbr[1]{ v \in H^1(\Omega^s)\ :\ v=0\ \text{on}\ \Gamma^s},
\end{equation}
where $C_{T,1}, C_{T,2}, C_{c,T} > 0 $ are independent of~$h_K$.

\subsection{Consistency}
\label{ss:consistency}

We now prove that the scheme in \cref{eq:hdgwf} is consistent with the
Stokes--Darcy system in \cref{eq:system,eq:interface}.

\begin{lemma}[Consistency]
  \label{lem:consistency}
  If $(u, p) \in X $ solves the Stokes--Darcy system
  \cref{eq:system,eq:interface}, then letting
  $\boldsymbol{u}=(u, \gamma(u))$ and
  $\boldsymbol{p} = (p, \gamma(p^s), \gamma(p^d))$,
  \begin{equation}
    B_h( (\boldsymbol{u}, \boldsymbol{p}), (\boldsymbol{v}, \boldsymbol{q}) )
    =
    \int_{\Omega^s} f^s \cdot v \dif x + \int_{\Omega^d} f^dq\dif x
    \quad
    \forall (\boldsymbol{v}, \boldsymbol{q}) \in \boldsymbol{X}(h).
  \end{equation}
\end{lemma}
\begin{proof}
  We consider each form in the definition of $B_h$ separately. Since $u
  = \gamma(u)$ on cell boundaries in $\overline{\Omega}^s$, and
  integration by parts,
  \begin{equation}
    \label{eq:ahconsistency_a}
    \begin{split}
      a_h^s(\boldsymbol{u}, \boldsymbol{v} ) &=
      \sum_{K\in\mathcal{T}^s} \int_K 2\mu \varepsilon(u) : \varepsilon(v) \dif x
      -\sum_{K\in\mathcal{T}^s} \int_{\partial K} 2\mu\varepsilon(u)n^s \cdot (v-\bar{v})\dif s
      \\
      &=
      -\sum_{K\in\mathcal{T}^s} \int_K 2\mu (\nabla\cdot\varepsilon(u))\cdot v \dif x
      +\sum_{K\in\mathcal{T}^s} \int_{\partial K} 2\mu\varepsilon(u)n^s \cdot \bar{v} \dif s.
    \end{split}
  \end{equation}
  Using smoothness of $u$, single-valuedness of $\bar{v}$, and
  \cref{eq:bc_I_slip}, we note that
  \begin{equation}
    \begin{split}
    \sum_{K\in\mathcal{T}^s} \int_{\partial K} 2\mu\varepsilon(u)n^s \cdot \bar{v} \dif s
    &= \int_{\Gamma^I} 2\mu\varepsilon(u)n^s \cdot \bar{v} \dif s
    \\
    &= -\int_{\Gamma^I} \alpha\kappa^{-1/2} (u^s)^t\cdot\bar{v}^t \dif s
    + \int_{\Gamma^I} \del{ 2\mu\del{n^s\cdot \varepsilon(u)n^s}n^s } \cdot \bar{v} \dif s.
    \end{split}
  \end{equation}
  Combining with \cref{eq:ahconsistency_a},
  \begin{multline}
    \label{eq:ahconsistency_b}
    a_h^s(\boldsymbol{u}, \boldsymbol{v})
    =
    -\sum_{K\in\mathcal{T}^s} \int_K 2\mu (\nabla\cdot\varepsilon(u))\cdot v \dif x
    -\int_{\Gamma^I} \alpha\kappa^{-1/2} (u^s)^t\cdot\bar{v}^t \dif s
    \\
    + \int_{\Gamma^I} \del{ 2\mu\del{n^s\cdot \varepsilon(u)n^s}n^s } \cdot \bar{v} \dif s.
  \end{multline}
  Applying integration by parts, noting that $\gamma(p^j) = p^j$ on cell
  boundaries in $\overline{\Omega}^j$, with $j=s,d$, and $n^d = -n^s$ on
  $\Gamma^I$, results in
  \begin{multline}
    \label{eq:bhconsistency_b}
    \sum_{j=s,d}\del{ b_h^j(\boldsymbol{p}^{j}, v) + b_h^{I,j}(\gamma(p^j), \bar{v}) }
    =
    \sum_{K\in\mathcal{T}^s}\int_{K} \nabla p \cdot v \dif x
    + \sum_{K\in\mathcal{T}^d}\int_{K} \nabla p \cdot v \dif x
    \\
    + \int_{\Gamma^I}(p^d-p^s)\bar{v}\cdot n^s \dif s.
  \end{multline}
  Next,
  \begin{equation}
    \label{eq:bhconsistency_uq_s}
    \begin{split}
      \sum_{j=s,d} \del{ b_h^j(\boldsymbol{q}^j, u) + b_h^{I,j}(\bar{q}^j , \gamma(u^s)) }
      =& -\sum_{K\in\mathcal{T}^s} \int_{K} q \nabla\cdot u \dif x
      +\sum_{K\in\mathcal{T}^s} \int_{\partial K} \bar{q}^s u \cdot n^s \dif s
      - \int_{\Gamma^I} \bar{q}^s u^s\cdot n^s \dif s
      \\
      &-\sum_{K\in\mathcal{T}^d} \int_{K} q \nabla\cdot u \dif x
      +\sum_{K\in\mathcal{T}^d} \int_{\partial K} \bar{q}^d u \cdot n^d \dif s
      - \int_{\Gamma^I} \bar{q}^d u^d\cdot n^d \dif s
      \\
      =& \int_{\Omega^d} q f^d \dif x,
    \end{split}
  \end{equation}
  where for the second equality we used \cref{eq:mass}-\cref{eq:bc_d},
  and that $\bar{q}^s$, $\bar{q}^d$, $u^s \cdot n^s$ and $u^d \cdot n^d$
  are single-valued on interior facets. Furthermore, note that
  \begin{equation}
    \label{eq:consistencyahID}
    a_h^I(\gamma(u^s), \bar{v}) = \int_{\Gamma^I} \alpha \kappa^{-1/2}\bar{v}^t \cdot (u^s)^t \dif s
    \quad\text{and}\quad
    a_h^d(u, v) = \int_{\Omega^d} \kappa^{-1} u \cdot v \dif x,
  \end{equation}
  hence summing
  \cref{eq:ahconsistency_b,eq:bhconsistency_b,eq:bhconsistency_uq_s,eq:consistencyahID}
  results in
  \begin{multline}
    B_h( (\boldsymbol{u}, \boldsymbol{p}), (\boldsymbol{v}, \boldsymbol{q}) )
    =
    \sum_{K\in\mathcal{T}^s} \int_K \del{-2\mu (\nabla\cdot\varepsilon(u)) + \nabla p }\cdot v \dif x
    + \sum_{K\in\mathcal{T}^d}\int_K \del{\kappa^{-1} u + \nabla p}\cdot v \dif x
    \\
    + \int_{\Gamma^I} \del{ 2\mu\del{n^s\cdot \varepsilon(u)n^s}n^s } \cdot \bar{v} \dif s
    + \int_{\Gamma^I}(p^d-p^s)n^s \cdot \bar{v} \dif s
    + \int_{\Omega^d} q f^d \dif x.
  \end{multline}
  The result follows after using
  \cref{eq:momentum,eq:d_velocity,eq:bc_I_p}. \qed
\end{proof}

\subsection{Coercivity and continuity of $a_h^s(\cdot, \cdot)$ and $a_h(\cdot, \cdot)$}

In this section we show that $a_h(\cdot, \cdot)$ is coercive on
$\boldsymbol{V}_h$ for sufficiently large penalty parameter $\beta$. We
furthermore prove continuity of $a_h^s(\cdot, \cdot)$ and $a_h(\cdot,
\cdot)$.

\begin{lemma}[Coercivity]
  \label{lem:coercivity_ahs_ah}
  There exists a constant $C > 0$, independent of $h$, and a constant
  $\beta_0 > 0$ such that for $\beta > \beta_0$ and for all
  $\boldsymbol{v}_h \in \boldsymbol{V}_h$,
  \begin{equation}
    \label{ineq:coer2}
    a_h(\boldsymbol{v}_h, \boldsymbol{v}_h) \ge C\tnorm{ \boldsymbol{v}_h }_v^2.
  \end{equation}
\end{lemma}
\begin{proof}
  Using identical steps as in \cite[Lemma~4.2]{Rhebergen:2017} and
  applying Korn's inequality \cite{Brenner:2004} it can be shown that
  \begin{equation}
    \label{ineq:coer1}
    a_h^s(\boldsymbol{v}_h, \boldsymbol{v}_h)
    \ge C \sum_{K\in \mathcal{T}^s}\del[1]{\|\varepsilon(v_h)\|_K^2+h_K^{-1}\|v_h-\bar{v}_h\|^2_{\partial K}}
    \ge C\tnorm{ \boldsymbol{v}_h }_{v,s}^2,
  \end{equation}
  The result follows by definition of $a_h$. \qed
\end{proof}

\begin{lemma}[Continuity]
  \label{lem:continuity}
  There exists a generic constant $C>0$, independent of $h$, such that
  for all $\boldsymbol{u}, \boldsymbol{v} \in \boldsymbol{V}(h)$,
  \begin{subequations}
    \begin{align}
      \label{ineq:continuity-ahs}
      a_h^s(\boldsymbol{u}, \boldsymbol{v}) & \le C \tnorm{\boldsymbol{u}}_{v',s}\tnorm{\boldsymbol{v}}_{v',s},
      \\ \label{ineq:continuity-ah}
      a_h(\boldsymbol{u}, \boldsymbol{v}) & \le C \tnorm{\boldsymbol{u}}_{v'}\tnorm{\boldsymbol{v}}_{v'}.
    \end{align}
    \label{ineq:continuity-ahall}
  \end{subequations}
\end{lemma}
\begin{proof}
  \Cref{ineq:continuity-ahs} follows by definition of $a_h^s$
  \cref{eq:ah_s}, the Cauchy-Schwarz inequality, the trace inequality
  \cref{eq:trace-2}, H\"{o}lder's inequality for sums and since
  $\envert[0]{\varepsilon(u)}_{H^{\ell}(K)} \leq
  C\envert{u}_{H^{\ell+1}(K)}$, $\ell = 0, 1$. \Cref{ineq:continuity-ah}
  follows by definition of $a_h$ \cref{eq:def_ah}, using the
  Cauchy--Schwarz inequality, \cref{ineq:continuity-ahs} and
  H\"{o}lder's inequality for sums. \qed
\end{proof}

\subsection{The inf-sup condition}

To present the inf-sup condition it will be convenient to split the
velocity-pressure coupling term $b_h^{j}(\boldsymbol{q}_h^j,v_h)$ in
\cref{eq:bh_j} as follows
\begin{equation}
  b_h^j(\boldsymbol{q}_h^j, v_h) = b_h^{j,1}(q_h, v_h ) +b_h^{j,2}(\bar{q}_h^j, v_h),
\end{equation}
where
\begin{equation*}
  b_h^{j,1}(q_h,v_h) =-\sum_{K\in \mathcal{T}^j}\int_Kq_h\nabla \cdot v_h\dif x
  \quad\text{and}\quad
  b_h^{j,2}(\bar{q}_h^j,v_h) =\sum_{K\in \mathcal{T}^j}\int_{\partial K}\bar{q}_h^jv_h\cdot n^j\dif s,
\end{equation*}
for $j=s, d$.

The main result of this section is the following theorem.
\begin{theorem}[inf-sup condition]
  \label{thm:infsup}
  There exists a constant $c_{{\rm inf}}^{\star} > 0$, independent of
  $h$, such that for any $\boldsymbol{q}_h \in \boldsymbol{Q}_h$,
  \begin{equation}
    c_{{\rm inf}}^{\star} \tnorm{\boldsymbol{q}_h}_{p} \le
    \sup_{\substack{\boldsymbol{v}_h \in \boldsymbol{V}_h \\ \boldsymbol{v}_h \ne 0}}
    \frac{\sum_{j=s,d}\del[1]{b_h^{j}(\boldsymbol{q}_h^j, v_h) +
        b_h^{I,j}(\bar{q}_h^j,\bar{v}_h)}}{\tnorm{\boldsymbol{v}_h}_{v}}.
  \end{equation}
\end{theorem}
To prove this result we require the definition of two interpolation
operators. For the velocity we require the following BDM interpolation
operator~\cite[Lemma~7]{Hansbo:2002}.
\begin{lemma}[BDM interpolation operator]
  \label{lem:BDM}
  If the mesh consists of triangles (${\rm dim}=2$) or tetrahedra (${\rm
  dim}=3$) there is an interpolation operator $\Pi_{V} :
  \sbr[0]{H^1(\Omega^j)}^{\rm dim} \rightarrow V_h^j$, where $V_h^j$ is
  the space of functions in $V_h$ restricted to cells in
  $\mathcal{T}^j$, such that for all $u \in \sbr[0]{H^{k+1}(K)}^{\rm
  dim}$,
  \begin{enumerate}[label=\roman*.]
  \item $\displaystyle \int_K q(\nabla \cdot (u-\Pi_V u)) \dif x =0$ for
    all $q\in P_{k-1}(K)$. \label{lem:BDM-i}

  \item $\displaystyle \int_F \bar{q} (u-\Pi_Vu) \cdot n \dif s =0$ for
    all $\bar{q}\in P_{k}(F)$, where $F$ is an edge (${\rm dim}=2$) or
    face (${\rm dim}=3$) of $\partial K$. \label{lem:BDM-ii}

    \item $\jump{n\cdot \Pi_V u}=0$, where $\jump{\cdot}$ is the usual
    jump operator. \label{lem:BDM-iii}

    \item $\norm{u-\Pi_V u}_{m,K} \le C h_K^{\ell-m}
    \norm{u}_{H^{\ell}(K)}$ with $m = 0, 1, 2$ and $m \le \ell \le k+1$.
    \label{lem:BDM-iv}
  \end{enumerate}
\end{lemma}
We also require an interpolation operator $\bar{\Pi}_V :
\sbr[0]{H^1(\Omega)}^{\rm dim} \to V_h\cap
\sbr[0]{C^0(\bar{\Omega})}^{\rm dim}$, for example the Scott--Zhang
interpolant \cite[Theorem~4.8.12]{brenner:book}, with the property that
for $v \in \sbr[0]{H^{\ell}(K)}^{\rm dim}$, $1 \leq \ell \leq k+1$,
\begin{subequations}
  \begin{align}
    \label{eq:l2projectionVbar_a}
    \norm[0]{v-\bar{\Pi}_Vv}_{\partial K} &\le C h_K^{\ell-1/2} \norm[0]{v}_{H^{\ell}(K)},
    \\
    \label{eq:l2projectionVbar_b}
    \norm[0]{\Pi_Vv-\bar{\Pi}_Vv}_{\partial K} &\le Ch_K^{\ell-1/2} \norm[0]{v}_{H^{\ell}(K)},
  \end{align}
  \label{eq:l2projectionVbar}
\end{subequations}
where $C$ is a generic constant independent of $h$.

Additionally, we require the following two auxiliary results.
\begin{lemma}
  \label{lem:infsup_bh1}
  There exists a constant $c_{{\rm inf}} > 0$, independent of $h$, such
  that for any $q_h\in Q_h$,
  \begin{equation}
    c_{{\rm inf}} \norm[0]{q_h}_{\Omega} \le
    \sup_{\substack{\boldsymbol{v}_h \in \boldsymbol{V}_h \\
    \boldsymbol{v}_h \ne 0}} \frac{\sum_{j=s,d}b_h^{j,1}(q_h,v_h)}{\tnorm{\boldsymbol{v}_h}_{v}}.
  \end{equation}
\end{lemma}
\begin{proof}
  Let $q_h \in Q_h$. It is well known,
  e.g.,~\cite[Theorem~6.5]{Pietro:book}, that since $q_h \in
  L^{2}_{0}(\Omega)$ there exists a $v \in [H_{0}^{1}(\Omega)]^{\rm
  dim}$ such that
  \begin{equation}
    \label{eq:infsup-1}
    -\nabla \cdot v=q_h
    \quad\text{and}\quad
    c_{{\rm inf}}\norm{v}_{H^1(\Omega)} \le \norm[0]{q_h}_{\Omega}.
  \end{equation}
  Then, by \cref{lem:BDM}~\eqref{lem:BDM-i}, since
  $q_h \in P_{k-1}(K)$ for all $K \in \mathcal{T}$,
  \begin{equation}
    \label{eq:infsup-2}
    \norm[0]{q_h}_{\Omega}^2 = -\int_{\Omega} q_h\nabla \cdot v \dif x
    = -\int_{\Omega} q_h \nabla \cdot \Pi_Vv \dif x.
  \end{equation}
  By \cref{lem:BDM}~\eqref{lem:BDM-iv} and \cref{eq:l2projectionVbar_b},
  \begin{equation}
    \label{eq:infsup-3}
      \tnorm{(\Pi_Vv, \bar{\Pi}_Vv)}_{v,s}^2
      =\sum_{K\in \mathcal{T}^s}\del[1]{\norm[0]{\nabla(\Pi_Vv)}_K^2
        + h_K^{-1}\norm[0]{\Pi_Vv-\bar{\Pi}_Vv}_{\partial K}^2}
      \le C\sum_{K\in \mathcal{T}^s} \norm[0]{ v }_{H^1(K)}^2 = C\norm[0]{v}^2_{H^1(\Omega^s)}.
  \end{equation}
  On the other hand, since $v \in \sbr[0]{H^1(\Omega^d)}^{\rm dim}$ such
  that $v = 0$ on $\Gamma^d$, using \cref{lem:BDM}~\eqref{lem:BDM-iv},
  \begin{equation}
    \label{eq:infsup-4}
    \norm{\Pi_Vv}_{\Omega^d}
    \le \norm{\Pi_Vv-v}_{\Omega^d} + \norm{v}_{\Omega^d}
    \le C \norm[0]{v}_{H^1(\Omega^d)}.
  \end{equation}
  Next, by \cref{eq:l2projectionVbar_a,eq:continuoustrace},
  \begin{multline}
    \label{eq:infsup-5}
    \norm[0]{(\bar{\Pi}_Vv)^t}_{\Gamma^I}^2
    \le \norm[0]{\bar{\Pi}_Vv}_{\Gamma^I}^2
    \le 2\del[1]{\norm[0]{\bar{\Pi}_Vv-v}_{\Gamma^I}^2+\norm{v}_{\Gamma^I}^2}
    \\
    \le 2\del[2]{\sum_{K\in \mathcal{T}^s} \norm[0]{\bar{\Pi}_Vv-v}^2_{\partial K}} + 2\norm{v}_{\Gamma^I}^2
    \le \del[2]{\sum_{K\in \mathcal{T}^s}Ch_K\norm[0]{v}^2_{H^1(K)}}+\norm{v}_{\Gamma^I}^2
    \le  C\norm{v}_{H^1(\Omega^s)}^2.
  \end{multline}
  Combining \cref{eq:infsup-3,eq:infsup-4,eq:infsup-5}, we obtain
  \begin{equation}
    \label{eq:infsup-6}
    \tnorm{(\Pi_Vv,\bar{\Pi}_Vv)}_v \le C\norm[0]{v}_{H^1(\Omega)}.
  \end{equation}
  Therefore, by \cref{eq:infsup-2,eq:infsup-6},
  \begin{equation}
    \sup_{\substack{\boldsymbol{v}_h\in \boldsymbol{V}_h \\ \boldsymbol{v}_h \ne 0}}\frac{-\int_{\Omega}q_h\nabla \cdot v_h\dif x}{\tnorm{\boldsymbol{v}_h}_{v}}
    \ge
    \frac{-\int_{\Omega}q_h\nabla \cdot \Pi_Vv\dif x}{\tnorm{(\Pi_Vv,\bar{\Pi}_V v)}_{v}}
    \ge
    \frac{\norm[0]{q_h}^2_{\Omega}}{C\norm[0]{v}_{H^1(\Omega)}}
    \ge
    \frac{c_{{\rm inf}}}{C}\norm[0]{q_h}_{\Omega},
  \end{equation}
  where we used \cref{eq:infsup-1} for the last inequality. \qed
\end{proof}

To prove the next auxiliary result, we introduce an operator to lift
$\bar{q}_h^s\in \bar{Q}_h^s$ and $\bar{q}_h^d \in \bar{Q}_h^d$ to
$\Omega^s$ and $\Omega^d$, respectively. Let $R_k(\partial K) :=
\cbr[0]{ \bar{q} \in L^2(\partial K),\ \bar{q}\in P_k(F),\ \forall F \in
\partial K}$ and let $L : R_k(\partial K) \rightarrow
\sbr[0]{P_k(K)}^{\rm dim}$ be the BDM local lifting operator which
satisfies for all $\bar{q}_h\in R_{k}(\partial K)$ the following:
\begin{equation}
  \label{eq:locallifting}
  (L\bar{q}_h)\cdot n = \bar{q}_h
  \quad\text{and}\quad
  \norm[0]{L\bar{q}_h}_K \le Ch_K^{1/2}\norm[0]{\bar{q}_h}_{\partial K}.
\end{equation}
Furthermore, it holds that
\begin{equation}
  \label{eq:lifting}
  \norm[0]{\nabla (L\bar{q}_h)}_K^2 \le Ch_K^{-1}\norm[0]{\bar{q}_h}^2_{\partial K}
  \quad\text{and}\quad
  \norm[0]{L\bar{q}_h}_{\partial K}^2 \le C\norm[0]{\bar{q}_h}^2_{\partial K}.
\end{equation}
See \cite[Example~2.5.1]{Boffi:book}.

\begin{lemma}
  \label{lem:infsup_bh2}
  There exists a constant $\bar{c}_{{\rm inf}} > 0$, independent of
  $h$, such that for any
  $(\bar{q}_h^s,\bar{q}_h^d)\in \bar{Q}_h^s\times \bar{Q}_h^d$,
  \begin{equation}
    \del[2]{\bar{c}_{{\rm inf}}\sum_{j=s,d}\sum_{K\in \mathcal{T}^j} h_K \norm[0]{\bar{q}_h^j}^2_{\partial K}}^{1/2}
    \le
    \sup_{\substack{\boldsymbol{v}_h\in \boldsymbol{V}_h \\ \boldsymbol{v}_h \ne 0}} \frac{\sum_{j=s,d} \del[1]{b_h^{j,2}(\bar{q}_h^j,v_h)
        + b_h^{I,j}(\bar{q}_h^j,\bar{v}_h)}}{\tnorm{\boldsymbol{v}_h}_{v}}.
  \end{equation}
\end{lemma}
\begin{proof}
  Let $\bar{q}_h^j \in \bar{Q}_h^j$ and set $w_h^j := L\bar{q}_h^j$
  for $j =s,d$. Since $w_h^j \in \sbr[0]{P_k(K)}^{\rm dim}$ for all
  $K \in \mathcal{T}^j$, $j = s, d$,
  then $w_h = w_h^s + w_h^d \in \sbr[0]{P_k(K)}^{\rm dim}$ for
  all $K \in \mathcal{T}$. Furthermore, by \cref{eq:lifting},
  \begin{equation}
    \label{eq:infsup-7}
    \tnorm{(w_h,0)}_{v,s}^2
    = \sum_{K\in \mathcal{T}^s}\del[2]{\norm[0]{\nabla (L\bar{q}_h^s)}_K^2 + h_K^{-1}\norm[0]{L\bar{q}_h^s}_{\partial K}^2}
    \le C\sum_{K\in \mathcal{T}^s}h_K^{-1}\norm[0]{\bar{q}_h^s}_{\partial K}^2.
  \end{equation}
  Next, using \cref{eq:locallifting},
  \begin{equation}
    \label{eq:infsup-8}
    \norm{w_h}_{\Omega^d}^2
    = \norm[0]{w_h^d}_{\Omega^d}^2
    = \sum_{K\in \mathcal{T}^d}\norm[0]{L\bar{q}_h^d}_K^2
    \le C\sum_{K\in \mathcal{T}^d}h_K\norm[0]{\bar{q}_h^d}_{\partial K}^2
    \le C\sum_{K\in \mathcal{T}^d}h_K^{-1}\norm[0]{\bar{q}_h^d}_{\partial K}^2,
  \end{equation}
  where we used $h_K < 1 < h_K^{-1}$. Then, combining
  \cref{eq:infsup-7,eq:infsup-8},
  \begin{equation}
    \tnorm{(w_h,0)}_v^2\leq C\sum_{j=s,d}\del[2]{\sum_{K\in \mathcal{T}^j}h_K^{-1}\norm[0]{\bar{q}_h^j}_{\partial K}^2}.
  \end{equation}
  Noting that $w_h^j\cdot n=\bar{q}_h^j$, $j=s,d$ we find
  \begin{align*}
    \sup_{\substack{\boldsymbol{v}_h \in \boldsymbol{V}_h \\ \boldsymbol{v}_h \ne 0}}\frac{\sum_{j=s,d}\del[1]{b_h^{j,2}(\bar{q}_h^j,v_h)
     + b_h^{I,j}(\bar{q}_h^j,\bar{v}_h)}}{\tnorm{\boldsymbol{v}_h}_{v}}
    &\ge\frac{\sum_{j=s,d}\del[2]{\sum_{K\in \mathcal{T}^j}\norm[0]{\bar{q}_h^j}_{\partial K}^2}}{\tnorm{(w_h,0)}_{v}}
    \\
    &\ge \frac{\sum_{j=s,d}\del[2]{\sum_{K\in \mathcal{T}^j}\norm[0]{\bar{q}_h^j}_{\partial K}^2}}
      {\del[2]{C\sum_{j=s,d}\del[2]{\sum_{K\in \mathcal{T}^j}h_K^{-1}\norm[0]{\bar{q}_h^j}_{\partial K}^2}}^{1/2}}
    \\
    &\ge C\del[2]{\tfrac{h_{\min}}{h_{\max}}}\del[2]{\sum_{j=s,d}\sum_{K\in \mathcal{T}^j}h_K\norm[0]{\bar{q}_h^j}_{\partial K}^2}^{1/2}.
  \end{align*}
  The result follows with
  $\bar{c}_{{\rm inf}} = C\del[2]{\tfrac{h_{\min}}{h_{\max}}}$. \qed
\end{proof}

Using \cref{lem:infsup_bh1} and \cref{lem:infsup_bh2}, the proof to
\cref{thm:infsup} then follows the steps as
\cite[Lemma~1]{Rhebergen:2018b} and is therefore omitted.

An immediate consequence of \cref{lem:coercivity_ahs_ah} and
\cref{thm:infsup} is existence and uniqueness of a solution to
\cref{eq:hdgwf}, as shown by the following proposition.

\begin{proposition}[Existence and uniqueness]
  \label{prop:existenceUniqueness}
  If $\beta>\beta_0$, there exists a unique solution $(\boldsymbol{u}_h,
  \boldsymbol{p}_h)\in \boldsymbol{X}_h$ to \cref{eq:hdgwf}.
\end{proposition}
\begin{proof}
  Since \cref{eq:hdgwf} is linear, it is sufficient to show uniqueness.
  Let $f^s = f^d = 0$. Let us take $\boldsymbol{v}_h = \boldsymbol{u}_h$
  and $\boldsymbol{q}_h = -\boldsymbol{p}_h$ in \cref{eq:hdgwf}. Then
  $a_h(\boldsymbol{u}_h, \boldsymbol{u}_h) = 0$ which, together with the
  coercivity result \cref{ineq:coer2}, implies that
  $\tnorm{\boldsymbol{u}_h}_{v} = 0$, that is, $\boldsymbol{u}_h = 0$.
  Inserting this into \cref{eq:hdgwf} with $f^s = f^d = 0$, we find
  \begin{equation}
    \label{eq:existence-1}
    \sum_{j=s,d}\del[1]{b_h^j(\boldsymbol{p}_h^j,v_h)+b_h^{I,j}(\bar{p}_h^j,\bar{v}_h)}=0
    \quad \forall \boldsymbol{v}_h \in \boldsymbol{V}_h.
  \end{equation}
  By the inf-sup condition in \cref{thm:infsup} this implies that
  $\boldsymbol{p}_h = 0$. \qed
\end{proof}

\section{Error analysis of the Stokes--Darcy system}
\label{sec:erroranalysisstokesdarcy}

In this section we present \emph{a priori} error analysis of the method
in \cref{eq:hdgwf}. For this we require the standard $L^2$-projection
operators onto $Q_h$ and $\bar{Q}_h^j$, $j = s, d$ which we denote,
respectively, by $\Pi_Q$ and $\bar{\Pi}_Q^j$. It can be shown,
e.g.~\cite{Pietro:book}, that for $k \ge 0$ and $0 \le \ell \le k$,
%
\begin{subequations}
  \begin{align}
    \norm[0]{ q-\Pi_Q q }_{K} & \le C h_K^{\ell} \norm[0]{ q }_{H^{\ell}(K)} && \forall q \in H^{\ell}(K),
    \\
    \norm[0]{ q-\bar{\Pi}_Q^j q }_{\partial K}
      & \le Ch_K^{\ell+1/2} \norm[0]{q}_{H^{\ell+1}(K)} && \forall q \in H^{\ell+1}(K),
  \end{align}
  \label{eq:l2projectionQQbar}
\end{subequations}
where $C$ is a generic constant independent of $h$.

It will be convenient to split the error into approximation and
interpolation errors:
\begin{alignat*}{4}
  u-u_h&=&\xi_u-\zeta_u, \quad \gamma(u)-\bar{u}_h&=\bar{\xi}_u-\bar{\zeta}_u,  \\
  p-p_h&=&\xi_p-\zeta_p, \quad \gamma(p^j)-\bar{p}_h^j&=\bar{\xi}_p^j-\bar{\zeta}_p^j, \quad j=s, d,
\end{alignat*}
where
\begin{alignat*}{8}
  \xi_u &:=&u - \Pi_V u,\quad \zeta_u&:=&u_h-\Pi_V u, \quad \bar{\xi}_u  &:=&\gamma(u^s)-\bar{\Pi}_{V} u^s,
  \quad \bar{\zeta}_u&:=&\bar{u}_h - \bar{\Pi}_V u^s, &\\
  \xi_p&:=&p-\Pi_Q p, \quad \zeta_p&:=&p_h-\Pi_Q p, \quad \bar{\xi}_p^j&:=&\gamma(p^j)-\bar{\Pi}_Q^j p^j,
  \quad \bar{\zeta}_p^j&:=&\bar{p}_h^j-\bar{\Pi}_Q^j p^j, & \quad j=s, d.
\end{alignat*}
To be consistent with our notation, we use $\boldsymbol{\xi}_u :=
(\xi_u,\bar{\xi}_u)$, $\boldsymbol{\xi}_p :=
(\xi_p,\bar{\xi}_p^s,\bar{\xi}_p^d)$ and $\boldsymbol{\xi}_p^{j} :=
(\xi_p,\bar{\xi}_p^j)$, $j=s,d$. Expressions $\boldsymbol{\zeta}_u$,
$\boldsymbol{\zeta}_p$ and $\boldsymbol{\zeta}_p^j$ are defined
similarly. We first present three results that will be useful in
following sections.
\begin{lemma}[$\tnorm{\cdot}_{v,s}$ and $\tnorm{\cdot}_v$-norm
  interpolation error estimates]
  \label{lem:vs}
  Suppose that $u$ is such that $u^s\in \sbr[0]{H^{\ell}(\Omega^s)}^{\rm
  dim}$, and $u^d\in \sbr[0]{H^{\ell-1}(\Omega^d)}^{\rm dim}$, $ 2\leq
  \ell\leq k+1$. Then
  \begin{subequations}
    \begin{align}
    \label{ineq:vs+}
      \tnorm{\boldsymbol{\xi}_u}_{v',s} &\le Ch^{\ell-1}\norm{u}_{H^{\ell}(\Omega^s)},
      \\
      \label{ineq:v+}
      \tnorm{\boldsymbol{\xi}_u}_{v'}  &\le Ch^{\ell-1}\del[1]{\norm[0]{u}_{H^{\ell}(\Omega^s)} + \norm[0]{u}_{H^{\ell-1}(\Omega^d)}}.
    \end{align}
  \end{subequations}
\end{lemma}
\begin{proof}
  Recall that
  \begin{equation}
    \label{eq:vs-1}
    \tnorm{\boldsymbol{\xi}_u}_{v',s}^2
    = \sum_{K\in \mathcal{T}^s}\del[1]{\norm[0]{\nabla \xi_u}_K^2 + h_K^{-1}\norm[0]{\xi_u-\bar{\xi}_u}_{\partial K}^2}
    + \sum_{K\in \mathcal{T}^s}h_K^2\envert[0]{\xi_u}_{2,K}^2.
  \end{equation}
  Then, since $\Pi_V u\in \sbr[0]{P_k(F)}^{\rm dim}$ for any $F\in
  \mathcal{F}^s$ we find by~\cref{eq:l2projectionVbar_b}
  \begin{equation}
    \label{eq:xiuxibaru}
    \norm[0]{\xi_u-\bar{\xi}_u}_{\partial K}^2
    = \norm[0]{\Pi_Vu-\bar{\Pi}_Vu}_{\partial K}^2
    \leq Ch_K^{2\ell-1} \norm[0]{u}_{H^{\ell}(K)}^2,
  \end{equation}
  \Cref{ineq:vs+} follows by combining \cref{eq:vs-1} and
  \cref{eq:xiuxibaru} and using once again
  \cref{lem:BDM}~\eqref{lem:BDM-iv}. \Cref{ineq:v+} follows from this
  and \cref{lem:BDM}~\eqref{lem:BDM-iv}. \qed
\end{proof}
\begin{lemma}[$\tnorm{\cdot}_{p,j}$-norm interpolation error]
  \label{lem:pj}
  For $j=s, d$, let $p^j\in H^{\ell}(\Omega^j)$, $0\leq \ell\leq k$. Then
  \begin{align*}
    \tnorm{\boldsymbol{\xi}_p^{j}}_{p,j}&\le Ch^{\ell} \|p\|_{H^{\ell}(\Omega^j)}.
  \end{align*}
\end{lemma}
\begin{proof}
  By definition of $\tnorm{\cdot}_{p,j}$ and the properties of the
  $L^2$-projections $\Pi_Q$ and $\bar{\Pi}_Q^j$ given in
  \cref{eq:l2projectionQQbar},
  \begin{equation*}
    \tnorm{\boldsymbol{\xi}_p^{j}}_{p,j}^2 = \norm[0]{\xi_p}_{\Omega^j}^2
    + \sum_{K\in \mathcal{T}^j} h_K \norm[0]{\bar{\xi}_p^j}_{\partial K}^2
    \leq Ch^{2\ell}\norm[0]{p}^2_{H^{\ell}(\Omega^j)}.
  \end{equation*}
  \qed
\end{proof}
\begin{lemma}[Error equation]
  \label{lem:erroreqn}
  There holds
  \begin{equation}
    B_h((\boldsymbol{\zeta}_u, \boldsymbol{\zeta}_p), (\boldsymbol{v}_h, \boldsymbol{q}_h)) =
    B_h((\boldsymbol{\xi}_u, \boldsymbol{\xi}_p), (\boldsymbol{v}_h, \boldsymbol{q}_h))
    \quad \forall (\boldsymbol{v}_h, \boldsymbol{q}_h) \in \boldsymbol{X}_h.
  \end{equation}
\end{lemma}
\begin{proof}
  Subtracting \cref{eq:hdgwf} from the consistency equation in
  \cref{lem:consistency} we obtain
  \begin{equation}\label{eq:erroreqn-1}
    B_h((u-u_h,\gamma(u^s)-\bar{u}_h,p-p_h,\gamma(p^s) - \bar{p}_h^s,\gamma(p^d)-\bar{p}_h^d),
      (\boldsymbol{v}_h, \boldsymbol{q}_h))=0,
  \end{equation}
  for all $(\boldsymbol{v}_h, \boldsymbol{q}_h) \in \boldsymbol{X}_h$.
  The result follows simply by splitting the errors. \qed
\end{proof}

\subsection{Energy-norm error estimates}

In this section we determine error estimates for the velocity and
pressure in the energy-norm.

\begin{theorem}[Energy-norm error estimates]
  \label{thm:energy_estimate}
  Let $(u,p)\in X$ be the solution of the Stokes--Darcy system
  \cref{eq:system,eq:interface} such that $u^s \in
  \sbr[0]{H^{k+1}(\Omega^s)}^{\rm dim}$ and $u^d \in
  \sbr[0]{H^k(\Omega^d)}^{\rm dim}$ for $k \ge 1$ and let
  $(\boldsymbol{u}_h, \boldsymbol{p}_h) \in \boldsymbol{X}_h$ solve
  \cref{eq:hdgwf}. Then
  \begin{subequations}
    \begin{equation}
      \label{eq:velocity_energy_estimate}
      \tnorm{\boldsymbol{u}-\boldsymbol{u}_h}_v
      \le Ch^{k} \del[1]{ \norm{u}_{H^{k+1}(\Omega^s)} + \norm{u}_{H^k(\Omega^d)} }.
    \end{equation}
    In addition, if $p \in H^k(\Omega^j)$, $j=s,d$, then
    \begin{equation}
      \label{eq:pressure_energy_estimate}
      \tnorm{\boldsymbol{p}-\boldsymbol{p}_h}_{p}
      \le
      Ch^k\del[1]{\norm{u}_{H^{k+1}(\Omega^s)} + \norm{u}_{H^k(\Omega^d)}
        + \norm[0]{p}_{H^k(\Omega^s)} + \norm[0]{p}_{H^k(\Omega^d)}}.
    \end{equation}
    \label{eq:velocitypressure_energy_estimate}
  \end{subequations}
\end{theorem}
\begin{proof}
  We will first prove
  \begin{equation}
    \label{eq:firstboundzetau}
    \tnorm{\boldsymbol{\zeta}_u}_v^2 \le c_1 \del{
      \tnorm{ \boldsymbol{\zeta}_u }_v + \tnorm{\boldsymbol{\zeta}_p}_p }\tnorm{ \boldsymbol{\xi}_u }_{v'},
  \end{equation}
  where $c_1 > 0$ is a constant independent of~$h$. Setting
  $(\boldsymbol{v}_h, \boldsymbol{q}_h) = (\boldsymbol{\zeta}_u,
  -\boldsymbol{\zeta}_p)$ in \cref{lem:erroreqn}, and using coercivity
  of $a_h$ \cref{ineq:coer2},
  \begin{equation}
    \label{eq:zetacoerc}
    B_h((\boldsymbol{\xi}_u, \boldsymbol{\xi}_p),(\boldsymbol{\zeta}_u, -\boldsymbol{\zeta}_p))
    =
    B_h((\boldsymbol{\zeta}_u, \boldsymbol{\zeta}_p),(\boldsymbol{\zeta}_u, -\boldsymbol{\zeta}_p))
    =
    a_h(\boldsymbol{\zeta}_u, \boldsymbol{\zeta}_u)
    \ge
    C\tnorm{\boldsymbol{\zeta}_u}_v^2.
  \end{equation}
 We proceed by bounding
 \begin{equation}
   \label{eq:Bh_to_I}
   \begin{split}
     B_h((\boldsymbol{\xi}_u, \boldsymbol{\xi}_p), (\boldsymbol{\zeta}_u, -\boldsymbol{\zeta}_p))
     =& a_h(\boldsymbol{\xi}_u, \boldsymbol{\zeta}_u)
     + \sum_{j=s,d}\del[1]{b_h^j(\boldsymbol{\xi}_p^{j}, \zeta_u) + b_h^{I,j}(\bar{\xi}_p^j, \bar{\zeta}_u)}
     -\sum_{j=s,d} \del[1]{b_h^j(\boldsymbol{\zeta}_p^{j},\xi_u)+b_h^{I,j}(\bar{\zeta}_p^j,\bar{\xi}_u)}
     \\
     =& I_1 +  I_2 + I_3.
   \end{split}
 \end{equation}
 Observe first that $I_2$ disappears by the properties of $\Pi_Q,
 \bar{\Pi}_Q^j$ and \cref{lem:BDM} \eqref{lem:BDM-i}-\eqref{lem:BDM-ii}.
 Consider now $I_1$. By \cref{lem:continuity} and equivalence of the
 norms $\tnorm{\cdot}_v$ and $\tnorm{\cdot}_{v'}$ on $\boldsymbol{V}_h$
 \begin{equation}
   \label{eq:I1}
   I_1 = a_h(\boldsymbol{\xi}_u, \boldsymbol{\zeta}_u)
   \le C\tnorm{ \boldsymbol{\zeta}_u }_v \tnorm{ \boldsymbol{\xi}_u }_{v'}.
 \end{equation}
 We next consider $I_3$ and note that by
 \cref{lem:BDM}~\eqref{lem:BDM-i} and~\eqref{lem:BDM-ii}
 \begin{equation*}
   \begin{split}
     I_3 &=
     - b_h^s(\boldsymbol{\zeta}_p^{s}, \xi_u) - b_h^{I,s}(\bar{\zeta}_p^s, \bar{\xi}_u)
     - b_h^d(\boldsymbol{\zeta}_p^{d}, \xi_u) - b_h^{I,d}(\bar{\zeta}_p^d, \bar{\xi}_u)
     \\
     &=
     - \int_{\Gamma^I}(\bar{p}_h^s - \bar{\Pi}^s_Qp^s)(u^s - \bar{\Pi}_Vu^s)\cdot n^s \dif s
     - \int_{\Gamma^I}(\bar{p}_h^d - \bar{\Pi}^d_Qp^d)(u^s - \bar{\Pi}_Vu^s)\cdot n^s \dif s.
   \end{split}
 \end{equation*}
 Since $(\bar{p}_h^j - \bar{\Pi}^j_Qp^j) \in P_k(F)$ for $j=s,d$, we
 have by \cref{lem:BDM}~\eqref{lem:BDM-ii} that
 \begin{equation}
   \label{eq:I3}
   \begin{split}
     I_3
     &=
     - \int_{\Gamma^I}(\bar{p}_h^s - \bar{\Pi}^s_Qp^s)(\Pi_Vu^s - \bar{\Pi}_Vu^s)\cdot n^s \dif s
     - \int_{\Gamma^I}(\bar{p}_h^d - \bar{\Pi}^d_Qp^d)(\Pi_Vu^s - \bar{\Pi}_Vu^s)\cdot n^s \dif s
     \\
     &=
     - \int_{\Gamma^I}\bar{\zeta}_p^s(\bar{\xi}_u - \xi_u) \cdot n^s \dif s
     - \int_{\Gamma^I}\bar{\zeta}_p^d(\bar{\xi}_u - \xi_u) \cdot n^d \dif s
     \\
     &\le
     \del{\sum_{K\in\mathcal{T}^s}h_k\norm[0]{\bar{\zeta}_p^s}_{\partial K}^2}^{1/2}\tnorm{\boldsymbol{\xi}_u}_{v,s}
     + \del{\sum_{K\in\mathcal{T}^d}h_k\norm[0]{\bar{\zeta}_p^d}_{\partial K}^2}^{1/2}\tnorm{\boldsymbol{\xi}_u}_{v,s}
     \\
     &\le
     C \del{ \tnorm{\boldsymbol{\zeta}_p^s}_{p,s} + \tnorm{\boldsymbol{\zeta}_p^d}_{p,d} }\tnorm{\boldsymbol{\xi}_u}_{v,s}
     \le C \tnorm{\boldsymbol{\zeta}_p}_p \tnorm{\boldsymbol{\xi}_u}_{v'}.
   \end{split}
 \end{equation}
 \Cref{eq:firstboundzetau} follows by combining \cref{eq:zetacoerc,eq:I3}.

 We next prove that
 \begin{equation}
   \label{eq:zetappnorm}
   \tnorm{\boldsymbol{\zeta}_p}_{p}^2
   \le c_2\del[1]{\tnorm{\boldsymbol{\zeta}_u}_{v}^2+\tnorm{\boldsymbol{\xi}_u}_{v'}^2},
 \end{equation}
 where $c_2 > 0$ is a constant independent of $h$. Setting
 $\boldsymbol{q}_h = 0$ in the error equation~\cref{lem:erroreqn}, using
 the projection properties of $\Pi_Q$ and $\bar{\Pi}_Q$,
 applying~\cref{ineq:continuity-ah} and using the equivalence of the
 norms $\tnorm{\cdot}_v$ and $\tnorm{\cdot}_{v'}$ on~$\boldsymbol{V}_h$,
 \begin{equation*}
   \sum_{j=s,d} \del[1]{b_h^{j}(\boldsymbol{\zeta}_p^{j},v_h) + b_h^{I,j}(\bar{\zeta}_p^j,v_h)}
   =
   -a_h(\boldsymbol{\zeta}_u, \boldsymbol{v}_h) + a_h(\boldsymbol{\xi}_u, \boldsymbol{v}_h)
   \le C\tnorm{\boldsymbol{\zeta}_u}_{v}\tnorm{\boldsymbol{v}_h}_{v}
    +C\tnorm{\boldsymbol{\xi}_u}_{v'}\tnorm{\boldsymbol{v}_h}_{v}.
 \end{equation*}
 \Cref{eq:zetappnorm} follows by \cref{thm:infsup}.

 We will now combine \cref{eq:firstboundzetau,eq:zetappnorm}. First
 note, that by Young's inequality, we may bound
 \cref{eq:firstboundzetau} as
 \begin{equation}
   \label{eq:firstboundzetau_2}
   \tnorm{\boldsymbol{\zeta}_u}_v^2 \le
   \frac{c_1\varepsilon}{2} \tnorm{ \boldsymbol{\zeta}_u }_v^2
   + \frac{c_1\delta}{2} \tnorm{\boldsymbol{\zeta}_p}_p^2
   + \frac{c_1}{2}\del{\varepsilon^{-1} + \delta^{-1}}\tnorm{ \boldsymbol{\xi}_u }_{v'}^2,
 \end{equation}
 for any $\varepsilon > 0$ and $\delta > 0$. Multiplying
 \cref{eq:firstboundzetau_2} by a constant $\beta > 0$ and adding to
 \cref{eq:zetappnorm},
 \begin{equation*}
   \tnorm{\boldsymbol{\zeta}_p}_{p}^2 + \beta\tnorm{\boldsymbol{\zeta}_u}_v^2
   \le
   \frac{\beta c_1\delta}{2} \tnorm{\boldsymbol{\zeta}_p}_p^2
   + \del{c_2 + \frac{\beta c_1\varepsilon}{2}} \tnorm{ \boldsymbol{\zeta}_u }_v^2
   + \del{c_2 + \frac{\beta c_1}{2}\del{\varepsilon^{-1} + \delta^{-1}}}\tnorm{ \boldsymbol{\xi}_u }_{v'}^2.
 \end{equation*}
 Choosing $\varepsilon < 2/c_1$, setting $\beta > c_2/(1 -
 c_1\varepsilon/2)$ and choosing $\delta < 2/(\beta c_1)$ we may write
 \begin{equation*}
   \label{eq:zetappnormzeta}
   \del{1 - \frac{\beta c_1\delta}{2}} \tnorm{\boldsymbol{\zeta}_p}_{p}^2
   + \del{\beta - c_2 - \frac{\beta c_1\varepsilon}{2} }\tnorm{\boldsymbol{\zeta}_u}_v^2
   \le
   \del{c_2 + \frac{\beta c_1}{2}\del{\varepsilon^{-1} + \delta^{-1}}}\tnorm{ \boldsymbol{\xi}_u }_{v'}^2,
 \end{equation*}
 resulting in
 \begin{equation}
   \label{eq:zetap_zetau_bounds}
   \tnorm{\boldsymbol{\zeta}_p}_{p} \le C \tnorm{ \boldsymbol{\xi}_u }_{v'}
   \quad \text{and} \quad
   \tnorm{\boldsymbol{\zeta}_u}_v \le C \tnorm{ \boldsymbol{\xi}_u }_{v'}.
 \end{equation}

 The energy-norm error estimates
 \cref{eq:velocitypressure_energy_estimate} now follow by the triangle
 inequality, the bounds in \cref{eq:zetap_zetau_bounds}, and the
 interpolation error estimates from \cref{lem:vs,lem:pj}.  \qed
\end{proof}

\subsection{$L^2$-norm error estimate for the velocity}

We will now determine $L^2$-norm error estimates for the velocity in the
Stokes and Darcy regions separately. To obtain these estimates we
consider the dual problem where $(U, P)$ is the solution to
\cref{eq:system,eq:interface} with $f^s = \Psi \in
\sbr[0]{L^2(\Omega^s)}^{\rm dim}$ and $f^d =
0$~\cite[eq.~(13)]{Girault:2014}. We will assume that this solution to
the dual problem satisfies the following regularity estimates
\begin{subequations}
  \begin{align}
    \|U\|_{H^{2}(\Omega^s)} &\leq C\|\Psi\|_{\Omega^s}, & \quad \|P\|_{H^{1}(\Omega^s)}&\leq C\|\Psi\|_{\Omega^s},
    \\
    \|U\|_{H^{1}(\Omega^d)} &\leq C\|\Psi\|_{\Omega^s}, & \quad \|P\|_{H^{2}(\Omega^d)}&\leq C\|\Psi\|_{\Omega^s},
  \end{align}
  \label{regularity}
\end{subequations}
which will allow us to prove the following result.
\begin{theorem}[$L^2$-norm error estimate for the velocity]
  \label{thm:L2error-vs-vd}
  Let $(u, p) \in X$ be the solution of \cref{eq:system,eq:interface}
  such that $u^s \in \sbr[0]{H^{k + 1}(\Omega^s)}^{\rm dim}$ and $u^d
  \in \sbr[0]{H^{k + 1}(\Omega^d)}^{\rm dim}$ for $k\ge 1$ and let
  $(\boldsymbol{u}_h, \boldsymbol{p}_h) \in \boldsymbol{X}_h$ solve
  \cref{eq:hdgwf}. Then
  \begin{subequations}
    \begin{align}
      \label{eq:L2error-vs}
      \norm{u-u_h}_{\Omega^s}
      &\le Ch^{k+1} \del[1]{ \norm[0]{u}_{H^{k+1}(\Omega^s)} + \norm[0]{u}_{H^k(\Omega^d)}
        + \norm[0]{f^d}_{H^{k}(\Omega^d)}},
      \\
      \label{eq:L2error-vd}
      \norm{u-u_h}_{\Omega^d}
      &\le Ch^{k+1} \del[1]{\|u\|_{H^{k+1}(\Omega^s)} + \|u\|_{H^{k+1}(\Omega^d)}
        + \|f^d\|_{H^k(\Omega^d)}}.
    \end{align}
  \end{subequations}
\end{theorem}
Different arguments will be used to prove the $L^2$-norm error estimates
for the velocity in the Stokes region \cref{eq:L2error-vs} and in the
Darcy region \cref{eq:L2error-vd}. We therefore consider the proofs of
these two inequalities separately.
\begin{proof}[of Stokes error estimate in \cref{eq:L2error-vs}] Let
  $(U,P)\in X$ be the solution of the dual problem where $\Psi := u -
  u_h$. Setting $\boldsymbol{U} := (U, \gamma(U^s))$, $\boldsymbol{P} :=
  (P, \gamma(P^s),\gamma(P^d))$, \cref{lem:consistency} implies
  \begin{equation*}
    B_h((\boldsymbol{U}, \boldsymbol{P}),(\boldsymbol{v}, \boldsymbol{q})) =
    \int_{\Omega^s} (u-u_h) \cdot v \dif x, \quad\forall (\boldsymbol{v}, \boldsymbol{q})\in \boldsymbol{X}(h).
  \end{equation*}
  Setting $\boldsymbol{v} = (u - u_h, \gamma(u^s)- \bar{u}_h)$,
  $\boldsymbol{q} = (p - p_h, \gamma(p^s)-\bar{p}_h^s, \gamma(p^d)-\bar{p}_h^d)$ in
  this equation,
  \begin{equation}
    \label{eq:L2errorvs-1}
    \norm{u-u_h}^2_{\Omega_s}
    =
    B_h\del[1]{ \del[0]{\boldsymbol{U}, \boldsymbol{P}},
      \del[0]{ \del[0]{ u-u_h,\gamma(u^s)-\bar{u}_h} , \del[0]{ p-p_h,\gamma(p^s)-\bar{p}_h^s, \gamma(p^d)-\bar{p}_h^d} } }.
  \end{equation}
  Next, setting $\boldsymbol{v}_h=(\Pi_V U, \bar{\Pi}_VU)$,
  $\boldsymbol{q}_h=(\Pi_QP,\bar{\Pi}_Q^sP,\bar{\Pi}_Q^dP)$ in
  \cref{eq:erroreqn-1}, we obtain
  \begin{multline}
    \label{eq:L2errorvs-2}
    B_h\del[1]{ \del[0]{ \del[0]{ u-u_h,\gamma(u^s)-\bar{u}_h} , \del[0]{p-p_h,\gamma(p^s)-\bar{p}_h^s,\gamma(p^d)-\bar{p}_h^d} },
       \del[0]{ \del[0]{ \Pi_VU ,\bar{\Pi}_VU}, \del[0]{\Pi_QP,\bar{\Pi}_Q^sP,\bar{\Pi}_Q^dP} } }
     \\
     = 0.
  \end{multline}
  Combining \cref{eq:L2errorvs-1,eq:L2errorvs-2} leads to
  \begin{equation}
    \label{eq:L2errorvs-3}
    \begin{split}
      \norm{u - u_h}^2_{\Omega_s}
      =&B_h\del[1]{ \del[0]{ \boldsymbol{\xi}_U, \boldsymbol{\xi}_P },
        \del[0]{ \del[0]{ u-u_h,\gamma(u^s)-\bar{u}_h}, \del[0]{p-p_h, \gamma(p^s)-\bar{p}_h^s, \gamma(p^d)-\bar{p}_h^d} } }
      \\
      =&a_h(\boldsymbol{\xi}_U, (u-u_h,\gamma(u^s)-\bar{u}_h))
      \\
      &+b_h^s(\boldsymbol{\xi}_P^{s},u-u_h)+b_h^{I,s}(\bar{\xi}_P^s,\gamma(u^s)-\bar{u}_h)
      \\
      &+b_h^d(\boldsymbol{\xi}_P^{d},u-u_h)+b_h^{I,d}(\bar{\xi}_P^d,\gamma(u^s)-\bar{u}_h)
      \\
      &+b_h^s((p-p_h, \gamma(p^s)-\bar{p}_h^s), \xi_U)+b_h^{I,s}(\gamma(p^s)-\bar{p}_h^s, \bar{\xi}_U)
      \\
      &+b_h^d((p-p_h,\gamma(p^d)-\bar{p}_h^d), \xi_U)+b_h^{I,d}(\gamma(p^d)-\bar{p}_h^d, \bar{\xi}_U)
      \\
      =&J_1+\hdots +J_5.
    \end{split}
  \end{equation}
  First, observe that $J_2$ vanishes due to \cref{eq:mass,eq:bc_s},
  single-valuedness of $u$ and
  \cref{eq:massconservation-1,eq:massconservations}. We will bound the
  remaining terms starting with $J_1$. Using \cref{ineq:continuity-ah},
  \begin{equation*}
    J_1\le C\tnorm{\boldsymbol{\xi}_U}_{v'}\tnorm{(u-u_h,\gamma(u^s)-\bar{u}_h)}_{v'}.
  \end{equation*}
  By
  \cref{ineq:vs+,eq:l2projectionVbar_a,} and~\cref{lem:BDM}~\eqref{lem:BDM-iv},
  \begin{multline}
    \label{eq:L2errorvs-4}
      \tnorm{\boldsymbol{\xi}_U}_{v'}^2
      = \tnorm{\boldsymbol{\xi}_U}_{v',s}^2 + \sum_{K\in \mathcal{T}^d}\norm[0]{\xi_U}_K^2
      + \sum_{F\in \mathcal{F}^I}\norm[0]{\bar{\xi}_U^t}_F^2
      \le Ch^{2}\del[1]{ \norm[0]{U}_{H^{2}(\Omega^s)}^2 + \norm[0]{U}_{H^{1}(\Omega^d)}^2 }
      \\
      \le Ch^{2}\norm[0]{u-u_h}_{\Omega^s}^2,
  \end{multline}
  where in the last step we used the regularity assumption
  \cref{regularity}. We find
  \begin{equation*}
    J_1 \le Ch\tnorm{(u - u_h, \gamma(u^s) - \bar{u}_h)}_{v'} \norm{u-u_h}_{\Omega^s}.
  \end{equation*}
  We next bound $J_3$. Recalling
  \cref{eq:mass,eq:massconservation-1,eq:massconservations}, smoothness
  of $u$ and \cref{eq:bc_d}, we have
  \begin{multline*}
      J_3
      = \sum_{K\in \mathcal{T}^d}\int_K\xi_P(f^d-\Pi_Qf^d)\dif x
      \le \norm[0]{\xi_P}_{\Omega^d}\norm[0]{f^d-\Pi_Qf^d}_{\Omega^d}
      \le \tnorm{\boldsymbol{\xi}_P^{d}}_{p,d} \norm[0]{f^d-\Pi_Qf^d}_{\Omega^d}
      \\
      \le Ch^{k+1}\norm[0]{P}_{H^{1}(\Omega^d)}\norm[0]{f^d}_{H^k(\Omega^d)}
      \le Ch^{k+1}\norm[0]{u-u_h}_{\Omega^s}\norm[0]{f^d}_{H^k(\Omega^d)},
  \end{multline*}
  where we used \cref{lem:pj,regularity}.

  For $J_4$, observe that since $\nabla \cdot U = 0$ and $\nabla
  \cdot \Pi_VU \in P_{k-1}(K)$,
  \begin{multline}
    \label{eq:L2errorvs-5}
    \int_K(p-p_h)\nabla \cdot \xi_{U}\dif x
    = \int_K(p-p_h)\nabla \cdot \Pi_VU\dif x
    = \int_K(\Pi_Q p-p_h)\nabla \cdot \Pi_VU\dif x
    \\
    = \int_K(\Pi_Q p-p_h)\nabla \cdot (U-\Pi_VU)\dif x
    = 0,
  \end{multline}
  where we used $\Pi_Qp - p_h \in P_{k - 1}(K)$ and
  \cref{lem:BDM}~\eqref{lem:BDM-i} in the last step. Using
  \cref{eq:L2errorvs-5} and noting that $U$ is smooth,
  \cref{eq:massconservation-2} holds, $\bar{\Pi}_VU$ is single-valued
  and $U=0$ (and so $\bar{\Pi}_VU = 0$) on $\Gamma^s$, we
  have
  \begin{multline*}
      J_4
      = \sum_{K\in \mathcal{T}^s}\int_{\partial K}(\gamma(p^s)-\bar{p}_h^s)(\xi_U-\bar{\xi}_U)\cdot n^s\dif s
      = \sum_{K\in \mathcal{T}^s}\int_{\partial K}\bar{\zeta}_p^s(\xi_U-\bar{\xi}_U)\cdot n^s\dif s
      \\
      \le \del[2]{\sum_{K\in \mathcal{T}^s}h_K\norm[0]{\bar{\zeta}_p^s}_{\partial K}^2}^{1/2}\tnorm{\boldsymbol{\xi}_U}_{v,s}
      \le Ch^{2}\tnorm{\boldsymbol{\zeta}_p^{s}}_{p,s}\norm{u-u_h}_{\Omega^s},
  \end{multline*}
  where we applied \cref{eq:L2errorvs-4}.

  Finally we bound $J_5$. Using \cref{eq:L2errorvs-5} and recalling that
  $U$ is smooth, \cref{eq:massconservation-2} holds, $\bar{\Pi}_VU$ is
  single-valued and $U\cdot n=0$ (and so $\Pi_VU\cdot n=0)$ on
  $\Gamma^d$,
  \begin{equation*}
    \begin{split}
      J_5
      & = \int_{\Gamma^I}(\gamma(p^d)-\bar{p}_h^d)(\xi_U-\bar{\xi}_U)\cdot n^d\dif s
      = \int_{\Gamma^I}\bar{\zeta}_p^d(\xi_U-\bar{\xi}_U)\cdot n^d\dif s
      \\
      & \le \del[2]{\sum_{K\in \mathcal{T}^d}h_K\norm[0]{\bar{\zeta}_p^d}_{\partial K}^2}^{1/2}
     \tnorm{\boldsymbol{\xi}_U}_{v,s} \le  Ch^{2}\tnorm{\boldsymbol{\zeta}_p^{d}}_{p,d}\norm[0]{u-u_h}_{\Omega^s},
    \end{split}
  \end{equation*}
  where we used \cref{eq:L2errorvs-4}. The result follows by combining
  the bounds for $J_1, \hdots, J_5$, cancelling a factor of
  $\norm{u - u_h}_{\Omega^s}$, using \cref{thm:energy_estimate} and
  the equivalence of $\tnorm{\cdot}_v$ and $\tnorm{\cdot}_{v'}$ on
  $\boldsymbol{V}_h$. \qed
\end{proof}

To prove \cref{eq:L2error-vd} we will treat the problem in the Darcy
region as a separate problem where the interface conditions are
treated as boundary conditions. We will first state a result which
constructs a vector function in
\begin{equation*}
  H^{\rm div}(\Omega^d) = \{v\in L^2(\Omega^d): \nabla \cdot v\in L^2(\Omega^d)\},
\end{equation*}
such that its normal component on the interface matches the normal
component of the error $u - u_h$ on the interface.
\begin{lemma}
  \label{lem:w}
  Assume that $u^s\in \sbr[0]{H^{k+1}(\Omega^s)}^{\rm dim}$ and $u^d\in
  \sbr[0]{H^{k}(\Omega^d)}^{\rm dim}$ with $k \ge 1$. Then there exists
  a function $w \in H^{\rm div}(\Omega^d)$ that satisfies $\nabla \cdot
  w=0$ in $\Omega^d$, $w\cdot n=0$ on $\Gamma^d$ and $w\cdot
  n=(u^s - u_h^s)\cdot n$ on $\Gamma^I$ such that
  \begin{equation}
    \label{ineq:w}
    \norm{w}_{\Omega^d}
    \le  C h^{k+1}(\norm[0]{u}_{H^{k+1}(\Omega^s)} + \norm[0]{u}_{H^k(\Omega^d)}
      + \norm[0]{f^d}_{H^{k}(\Omega^d)}),
  \end{equation}
  where $C > 0$ is a constant independent of $h$.
\end{lemma}
We point out that the proof of \cref{lem:w} relies on pointwise mass
conservation \cref{eq:massconservation-1}, $H(\text{div})$-conformity
\cref{eq:massconservations} and~\cref{eq:L2error-vs}, and can be found
in~\cite[Lemma 3.1, Lemma 3.2]{Girault:2014}. We now prove
\cref{eq:L2error-vd}.
\begin{proof}[of the Darcy error estimate in \cref{eq:L2error-vd}]
  In \cref{eq:erroreqn-1} set $v_h=0$ in $\Omega^s$, $\bar{v}_h = 0$ and
  $\boldsymbol{q}_h = 0$. Then
  \begin{equation}
    \label{thm:L2errorvd-1}
    a_h^d(u-u_h,v_h) = -b_h^d((p-p_h, \gamma(p) - \bar{p}_h^d), v_h)
     = -b_h^d((\Pi_Q p - p_h,\bar{\Pi}_Q p - \bar{p}_h), v_h)
  \end{equation}
  by the projection properties of $\Pi_Q$ and $\bar{\Pi}_Q$.

  We will first determine a $v_h$ so that $a_h^d((u - u_h, v_h)) = 0$.
  Setting $q_h = 0$ in $\Omega^s$, $\bar{q}_h^s = 0$ and
  $\boldsymbol{v}_h = 0$ in \cref{eq:erroreqn-1}, using
  \cref{lem:BDM}~\eqref{lem:BDM-i} and \eqref{lem:BDM-ii}, we find
  \begin{equation}
    \label{eq:bhdbhId_1}
    b_h^d(\boldsymbol{q}_h^{d},\Pi_Vu-u_h)
    = b_h^d((q_h,\bar{q}_h^d), u-u_h)
    = \int_{\Gamma_I}\bar{q}_h^d\del{\gamma(u^d) - u_h^d}\cdot n^d \dif s,
  \end{equation}
  where the second equality is by
  \cref{eq:mass,eq:massconservation-1}. Next, let
  $w\in H^{\rm div}(\Omega^d)$ such that $\nabla\cdot w=0$ in
  $\Omega^d$, $w\cdot n=0$ on $\Gamma^d$,
  $w\cdot n^d= (\gamma(u^s)-u_h^s)\cdot n^d$ on $\Gamma^I$ as defined
  in \cref{lem:w}. Then, by \cref{lem:BDM}~\eqref{lem:BDM-i} and
  \eqref{lem:BDM-ii}
  \begin{equation}
    \label{eq:bhdbhId_2}
    b_h^d(\boldsymbol{q}_h^{d}, \Pi_Vw)
    = b_h^d(\boldsymbol{q}_h^{d},w)
    =\int_{\Gamma^I}\bar{q}_h^d(\gamma(u^s)-u_h^s)\cdot n^d \dif s,
  \end{equation}
  where the second equality is by
  \cref{eq:bc_d,eq:massconservation-2}. Combining \cref{eq:bhdbhId_1}
  and \cref{eq:bhdbhId_2}, and using
  \cref{eq:bc_I_u,eq:massconservation-4}, we find that
  \begin{equation}
    b_h^d(\boldsymbol{q}_h^{d},\Pi_Vu-u_h-\Pi_Vw)=0 \quad \forall \boldsymbol{q}_h^{d} \in \boldsymbol{Q}_h^d.
  \end{equation}
  Therefore, setting $v_h=\Pi_Vu-u_h-\Pi_Vw$ in
  \cref{thm:L2errorvd-1}, we find
  \begin{equation}
    \label{eq:ahdzerovh}
    0=a_h^d(u-u_h,\Pi_Vu-u_h-\Pi_Vw)=a_h^d(u-u_h,(\Pi_Vu-u)+(u-u_h)-\Pi_V w).
  \end{equation}
  Next, by definition of $a_h^d(\cdot, \cdot)$ and using
  \cref{eq:ahdzerovh}, we find that
  \begin{equation*}
    \begin{split}
      \kappa^{-1}\norm{u-u_h}_{\Omega^d}^2
      &= a_h^d(u-u_h,u-u_h)
      = a_h^d(u-u_h,u-\Pi_Vu+\Pi_Vw)
      \\
      & \le \kappa^{-1}\norm{u-u_h}_{\Omega^d}\norm{u-\Pi_Vu+\Pi_Vw}_{\Omega^d}.
    \end{split}
  \end{equation*}
  Hence, cancelling a factor of $\kappa^{-1}\|u-u_h\|_{\Omega^d}$ above
  and using \cref{lem:BDM}~\eqref{lem:BDM-iv} and \cref{ineq:w},
  \begin{equation*}
    \begin{split}
      \norm{u-u_h}_{\Omega^d}
      &\le \norm{u-\Pi_Vu}_{\Omega^d} + \norm{\Pi_Vw}_{\Omega^d}
      \le \norm{u-\Pi_Vu}_{\Omega^d} + C\norm{w}_{\Omega^d}
      \\
      &\le Ch^{k+1}\norm{u}_{H^{k+1}(\Omega^d)}
      + Ch^{k+1}\del[1]{\norm[0]{u}_{H^{k+1}(\Omega^s)} + \norm{u}_{H^{k}(\Omega^d)} + \norm[0]{f^d}_{H^k(\Omega^d)}}
      \\
      &\le Ch^{k+1}\del[1]{\norm[0]{u}_{H^{k+1}(\Omega^s)} + \norm{u}_{H^{k+1}(\Omega^d)} + \norm[0]{f^d}_{H^k(\Omega^d)}}.
    \end{split}
  \end{equation*}
  The result follows. \qed
\end{proof}

We note that the error estimates for the velocity in
\cref{thm:energy_estimate} and \cref{thm:L2error-vs-vd} do not depend
on the approximation error of the pressure. The EDG--HDG method
\cref{eq:hdgwf} for the Stokes--Darcy system is therefore
\emph{pressure-robust} \cite{John:2017}. We furthermore remark that
the analysis in this paper is easily extended to spatially dependent
permeability that varies continuously over $\Omega^d$ under the
condition that
$0 < \kappa_{\min} \le \kappa \le \kappa_{\max} < \infty$.

\section{Numerical examples}
\label{sec:numerical_examples}

The examples in this section have been implemented using the NGSolve
finite element library~\citep{Schoberl:2014}. Both examples are posed
on the domain $\Omega = [0, 1]^2$ with
$\Omega^d = [0, 1]\times[0, 0.5]$ and
$\Omega^s = [0, 1]\times [0.5, 1]$. The domain is discretized by an
unstructured simplicial mesh. The mesh is such that $\mathcal{T}^s$ is
an exact triangulation of $\Omega^s$, $\mathcal{T}^d$ is an exact
triangulation of $\Omega^d$, and cell facets match on the interface
$\Gamma^I$. We set $k=3$ in \cref{eq:DGcellspaces} and
\cref{eq:DGfacetspaces} and set the penalty parameter to be
$\beta = 10k^2$.

\subsection{Rates of convergence}
\label{ss:testing_rates_convergence}
We solve the Stokes--Darcy system \cref{eq:system,eq:interface} with the
source terms and boundary conditions chosen such that the exact solution
is given by
\begin{subequations}
  \label{eq:exactsolution}
  \begin{align}
    u|_{\Omega^s} &=
      \begin{bmatrix}
        -\sin(\pi x_1)\exp(x_2/2)/(2\pi^2)
        \\
        \cos(\pi x_1)\exp(x_2/2)/\pi
      \end{bmatrix},
    &
      p|_{\Omega^s} &= \frac{\kappa\mu-2}{\kappa\pi}\cos(\pi x_1)\exp(x_2/2),
    \\
    u|_{\Omega^d} &=
                    \begin{bmatrix}
                      -2\sin(\pi x_1)\exp(x_2/2)
                      \\
                      \cos(\pi x_1)\exp(x_2/2)/\pi
                    \end{bmatrix},
                  &
                    p|_{\Omega^d} &= -\frac{2}{\kappa\pi}\cos(\pi x_1)\exp(x_2/2),
  \end{align}
\end{subequations}
with $\alpha = \mu \kappa^{1/2}(1 + 4\pi^2)/2$. For $\mu = 1$ and
$\kappa = 1$ we recover the test case introduced in~\cite{Correa:2009}.
We introduce this extended version to investigate the effect of $\mu$
and $\kappa$ on the solution.

\Cref{tab:ratesconvpresrob_p3} presents computed errors in the
velocity and the pressure on $\Omega^s$ and on $\Omega^d$ for
different values of $\mu$ and $\kappa$.  We observe optimal rates of
convergence for all unknowns, independent of the value of the
viscosity and permeability. We observe also that, although the error
in the pressure changes when viscosity and permeability changes, the
error in the velocity remains the same, as predicted by the
analysis. Furthermore, pointwise satisfaction of the mass equation is
obtained in both regions. Although not shown here, these observations
hold also for other values of $k$ in \cref{eq:DGcellspaces} and
\cref{eq:DGfacetspaces}.

\begin{table}
  \small
  \centering {
    \begin{tabular}{cccccc}
      \hline\hline
      \multicolumn{6}{c}{Stokes solution ($\Omega^s$)} \\
      \hline
      Cells
      & $\norm{u_h - u}_{\Omega^s}$ & Rate & $\norm{p_h - p}_{\Omega^s}$
      & Rate & $\norm{\nabla \cdot u_h}_{\Omega^s}$ \\
      \hline
      \hline
      \multicolumn{6}{l}{$\mu=1$, $\kappa=1$} \\
   152 & 1.7e-6 & 4.8 & 2.9e-4 & 3.6 & 2.3e-14 \\
   578 & 8.4e-8 & 4.3 & 3.6e-5 & 3.0 & 4.8e-14 \\
  2416 & 4.4e-9 & 4.3 & 4.5e-6 & 3.0 & 9.3e-14 \\
  9580 & 2.3e-10 & 4.3 & 5.6e-7 & 3.0 & 1.9e-13 \\
      \multicolumn{6}{l}{$\mu=10^{-6}$, $\kappa=1$} \\
   152 & 9.0e-6 & 7.6 & 5.1e-5 & 3.2 & 7.5e-12 \\
   578 & 1.2e-7 & 6.3 & 6.1e-6 & 3.1 & 7.0e-12 \\
  2416 & 4.5e-9 & 4.7 & 7.4e-7 & 3.0 & 7.0e-12 \\
  9580 & 2.3e-10 & 4.3 & 9.0e-8 & 3.0 & 6.9e-12 \\
      \multicolumn{6}{l}{$\mu=1$, $\kappa=10^3$} \\
   152 & 1.7e-6 & 4.8 & 2.9e-4 & 3.6 & 2.3e-14 \\
   578 & 8.4e-8 & 4.3 & 3.6e-5 & 3.0 & 4.8e-14 \\
  2416 & 4.4e-9 & 4.3 & 4.5e-6 & 3.0 & 9.3e-14 \\
  9580 & 2.3e-10 & 4.3 & 5.6e-7 & 3.0 & 1.9e-13 \\
      \multicolumn{6}{l}{$\mu=10^{-6}$, $\kappa=10^3$} \\
   152 & 1.6e-6 & 4.8 & 5.1e-8 & 3.2 & 1.3e-14 \\
   578 & 8.3e-8 & 4.3 & 6.1e-9 & 3.1 & 1.4e-14 \\
  2416 & 4.4e-9 & 4.3 & 7.4e-10 & 3.0 & 5.6e-14 \\
  9580 & 2.3e-10 & 4.3 & 9.0e-11 & 3.0 & 4.9e-14 \\
      \hline\hline
      \multicolumn{6}{c}{Darcy solution ($\Omega^d$)} \\
      \hline
      Cells
      & $\norm{u_h - u}_{\Omega^d}$ & Rate & $\norm{p_h - p}_{\Omega^d}$
      & Rate & $\norm{\nabla \cdot u_h+\Pi_Q f^d}_{\Omega^d}$ \\
      \hline
      \hline
      \multicolumn{6}{l}{$\mu=1$, $\kappa=1$} \\
   152 & 3.1e-6 & 4.7 & 3.5e-5 & 3.4 & 1.2e-12 \\
   578 & 2.0e-7 & 4.0 & 4.7e-6 & 2.9 & 3.7e-12 \\
  2416 & 1.2e-8 & 4.1 & 5.8e-7 & 3.0 & 1.5e-11 \\
  9580 & 7.5e-10 & 4.0 & 7.3e-8 & 3.0 & 6.0e-11 \\
      \multicolumn{6}{l}{$\mu=10^{-6}$, $\kappa=1$} \\
   152 & 3.0e-6 & 4.7 & 3.5e-5 & 3.4 & 9.4e-13 \\
   578 & 2.0e-7 & 3.9 & 4.7e-6 & 2.9 & 3.6e-12 \\
  2416 & 1.2e-8 & 4.1 & 5.8e-7 & 3.0 & 1.5e-11 \\
  9580 & 7.5e-10 & 4.0 & 7.3e-8 & 3.0 & 5.9e-11 \\
      \multicolumn{6}{l}{$\mu=1$, $\kappa=10^3$} \\
   152 & 3.1e-6 & 4.7 & 3.5e-8 & 3.4 & 2.1e-12 \\
   578 & 2.0e-7 & 4.0 & 4.7e-9 & 2.9 & 3.6e-12 \\
  2416 & 1.2e-8 & 4.1 & 5.8e-10 & 3.0 & 1.0e-10 \\
  9580 & 7.5e-10 & 4.0 & 7.3e-11 & 3.0 & 6.2e-11 \\
      \multicolumn{6}{l}{$\mu=10^{-6}$, $\kappa=10^3$} \\
   152 & 3.1e-6 & 4.7 & 3.5e-8 & 3.4 & 1.9e-12 \\
   578 & 2.0e-7 & 3.9 & 4.7e-9 & 2.9 & 3.6e-12 \\
  2416 & 1.2e-8 & 4.1 & 5.8e-10 & 3.0 & 1.4e-10 \\
  9580 & 7.5e-10 & 4.0 & 7.3e-11 & 3.0 & 6.3e-11 \\
      \hline
    \end{tabular}
  } \caption{Errors and rates of convergence in $\Omega^s$ (top half)
    and $\Omega^d$ (bottom half) for the velocity and pressure fields
    using polynomial degree $k=3$ for the test case described in
    \cref{ss:testing_rates_convergence}.}
  \label{tab:ratesconvpresrob_p3}
\end{table}

\subsection{Coupled surface and subsurface flow}
\label{ss:surface_subsurface}

In this section we consider a test case similar to that proposed
in~\cite[Example~7.2]{Vassilev:2009}. This test case is representative
of surface flow coupled to subsurface flow.

Let the boundary of the Stokes region be partitioned as
$\Gamma^s = \Gamma^s_1\cup\Gamma_2^s\cup\Gamma_3^s$ where
$\Gamma^s_1 :=\cbr{ x \in\Gamma^s\ :\ x_1=0 }$,
$\Gamma^s_2 :=\cbr{ x \in\Gamma^s\ :\ x_1=1 }$ and
$\Gamma^s_3 :=\cbr{ x \in\Gamma^s\ :\ x_2=1 }$. Similarly, let
$\Gamma^d = \Gamma^d_1\cup\Gamma_2^d$ where
$\Gamma^d_1 :=\cbr{ x \in\Gamma^d\ :\ x_1=0 \ \text{or} \ x_1=1}$ and
$\Gamma^d_2 :=\cbr{ x \in\Gamma^d\ :\ x_2=0 }$. We impose the
following boundary conditions:
\begin{align*}
  u &= (x_2(3/2-x_2)/5, 0) && \text{on}\ \Gamma_1^s, \\
  \del{-2\mu\varepsilon(u) + p\mathbb{I}}n &= 0 && \text{on}\ \Gamma_2^s, \\
  u\cdot n &= 0 \ \text{and}\ \del{-2\mu\varepsilon(u) + p\mathbb{I}}^t = 0 && \text{on}\ \Gamma_3^s, \\
  u\cdot n &= 0 && \text{on}\ \Gamma_1^d, \\
  p &= -0.05 && \text{on}\ \Gamma_2^d.
\end{align*}
Note that due to the boundary condition on the pressure on $\Gamma_2^d$
the pressure solution need not be in $L^2_0(\Omega)$ as imposed in
\cref{eq:DGcellspaces}. We set the permeability to
\begin{equation*}
  \kappa = 700 (1 + 0.5 (\sin(10\pi x_1)\cos(20\pi x_2^2)
    + \cos^2(6.4 \pi x_1)\sin(9.2\pi x_2))) + 100.
\end{equation*}
Other parameters in \cref{eq:system,eq:interface} are set as $\mu =
0.1$, $\alpha=0.5$, $f^s=0$ and $f^d=0$. We consider the solution on a
mesh with 5754 simplicial cells.

\Cref{fig:contaminant_transport_conduc_vel} shows the permeability,
computed velocity field, and computed pressure field. We observe a
similar flow field as presented in~\cite[Example~7.2]{Vassilev:2009}. In
the Darcy region $\Omega^d$ the flow field avoids areas of low
permeability, free-flow is present in the Stokes region $\Omega^s$,
while the tangential velocity is discontinuous along the interface. To
plot the pressure we used two colour scales to illustrate the pressure
differences in the Darcy and the Stokes region. We note that the
pressure is discontinuous across the interface. Furthermore, the
pressure in the Darcy region follows a similar pattern as the
permeability while in the Stokes region pressure is highest at the
inlet.

\begin{figure}
  \centering
  \subfloat[The permeability. \label{fig:tc_contaminant_hc}]{\includegraphics[width=0.48\textwidth]{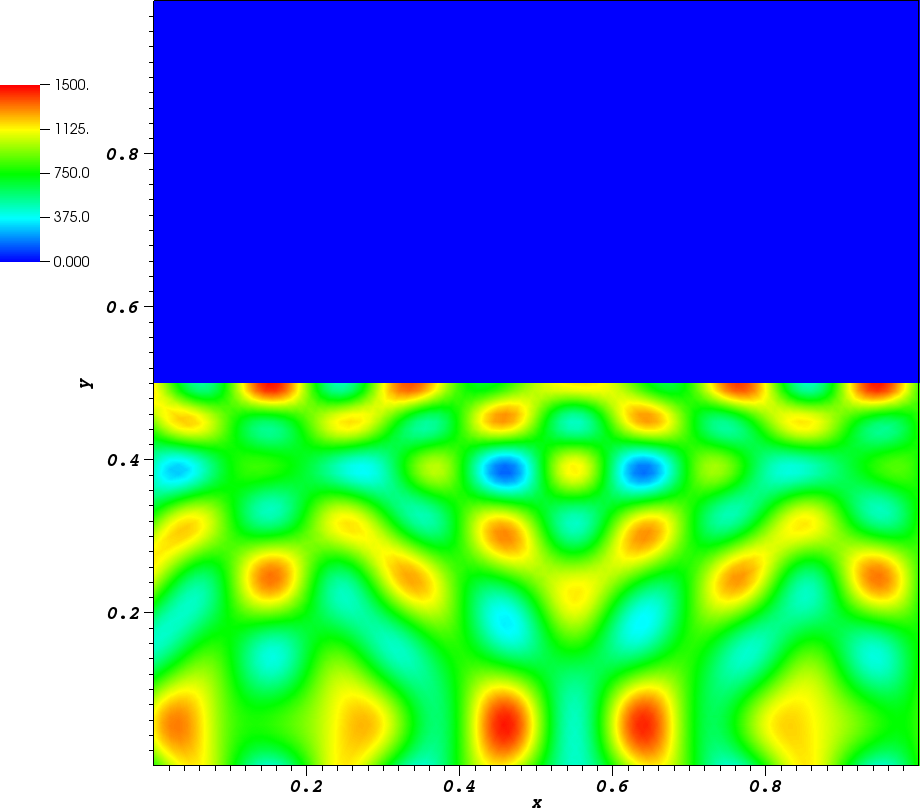}}
  \quad
  \subfloat[The velocity field. \label{fig:tc_contaminant_vel}]{\includegraphics[width=0.48\textwidth]{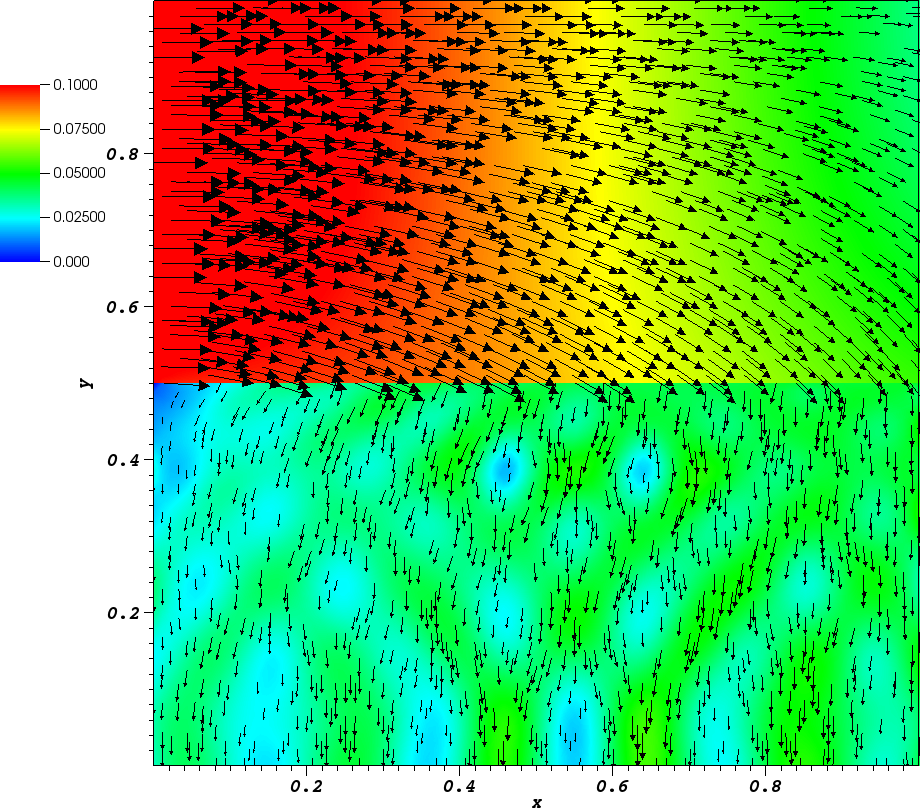}}
  \\
  \subfloat[The pressure field. \label{fig:tc_contaminant_pres}]{\includegraphics[width=0.48\textwidth]{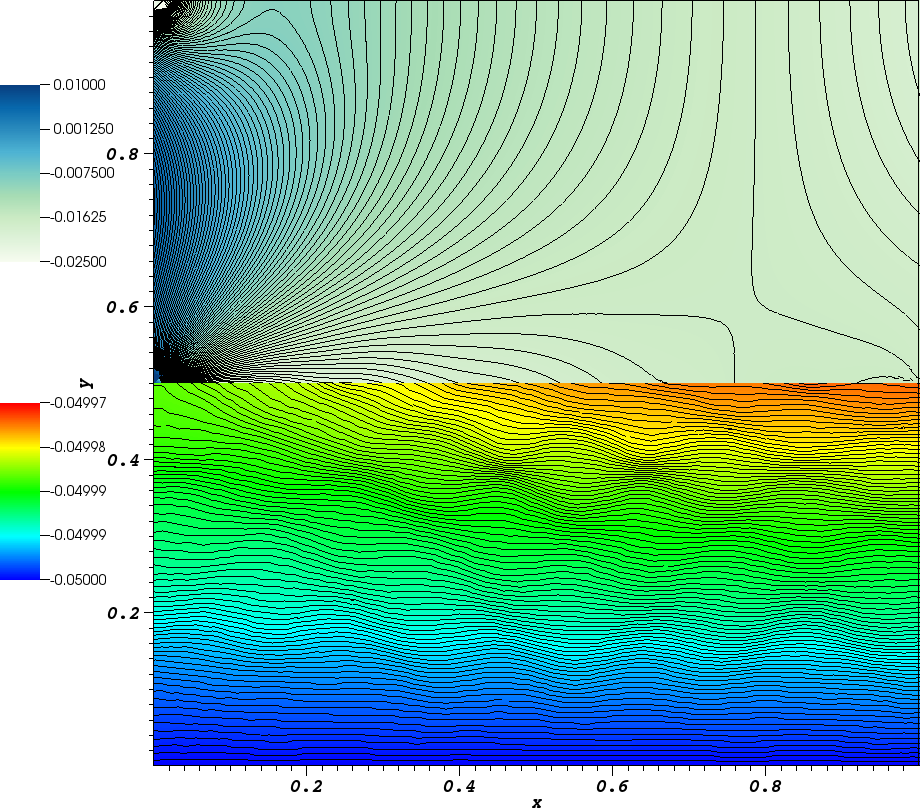}}
  \caption{The permeability, computed velocity field, and computed
    pressure field for the test case described
    in~\cref{ss:surface_subsurface}.}
  \label{fig:contaminant_transport_conduc_vel}
\end{figure}

\section{Conclusions}
\label{sec:conclusions}

We have formulated a pointwise mass-conserving and divergence-conforming
embedded--hybridized discontinuous Galerkin method for the Stokes--Darcy
system. The Stokes equations are coupled to the Darcy equations at an
interface by the Beavers--Joseph--Saffman condition. This coupling is
handled naturally by the facet variables that occur in the EDG--HDG
method and act as Lagrange multipliers enforcing divergence-conformity
of the scheme. We analyzed the embedded--hybridized method and proved
existence and uniqueness of the solution to the discretization.
Additionally, we presented \emph{a priori} error analysis showing
optimal rates of convergence and demonstrated that the error in the
velocity is independent of the pressure. Numerical examples support the
analysis of the scheme.

\subsubsection*{Acknowledgements}

SR gratefully acknowledges support from the Natural Sciences and
Engineering Research Council of Canada through the Discovery Grant
program (RGPIN-05606-2015) and the Discovery Accelerator Supplement
(RGPAS-478018-2015). Part of this research was done while AC was on
sabbatical at the University of Waterloo, Department of Applied
Mathematics.

\bibliographystyle{elsarticle-num-names}
\bibliography{references}
\end{document}